\documentclass[12pt]{amsart}

\usepackage{latexsym,enumerate,graphicx}
\usepackage{amsmath,amsthm,amsopn,amstext,amscd,amsfonts,amssymb}
\usepackage{fullpage,yhmath}
\usepackage{hyperref}
\usepackage{eucal}
\usepackage{color}
\usepackage{verbatim}

\numberwithin{equation}{section}

\newtheorem{theorem}{Theorem}
\newtheorem{lemma}{Lemma}

\newtheorem{proposition}{Proposition}
\newtheorem{remark}{Remark}

\newtheorem{example}{Example}

\numberwithin{theorem}{section}
\numberwithin{corollary}{section}
\numberwithin{lemma}{section}
\numberwithin{definition}{section}
\numberwithin{proposition}{section}
\numberwithin{remark}{section}

\newcommand{\R}{\mathbb R}
\newcommand{\N}{\mathbb N}

\newcommand{\medint}{-\kern  -,375cm\int}
\newcommand{\dint}{\displaystyle\int}

\newcommand{\dist}{\mathop{\mathrm{dist}}\nolimits}

\newcommand{\arcsinh}{\mathop{\mathrm{arcsinh}}\nolimits}
\newcommand{\J}{\mathcal{J}}

\begin{document}
\title{Sharp Poincar\'e inequalities in a class of non-convex sets}
\author{B. Brandolini$^\dagger$- F. Chiacchio$^\dagger$ - E. B. Dryden$^\ddagger$ - J. J. Langford$^\ddagger$}

\date{}

\maketitle

\footnotesize{
\centerline{$^\dagger$ Dipartimento di Matematica e Applicazioni ``R. Caccioppoli''}
\centerline{Universit\`a degli Studi di Napoli Federico II}
\centerline{Monte S. Angelo, via Cintia, I-80126 Napoli, Italy}
\centerline{\texttt{ brandolini@unina.it}}
\centerline{\texttt{francesco.chiacchio@unina.it}}

\bigskip
\centerline{$^\ddagger$ Department of Mathematics, Bucknell University}
\centerline{One Dent Drive, Lewisburg, PA 17837, USA}
\centerline{\texttt{emily.dryden@bucknell.edu}}
\centerline{\texttt{jeffrey.langford@bucknell.edu}}
}

\begin{abstract}
Let $\gamma$ be a smooth, non-closed, simple curve whose image is symmetric with respect to the $y$-axis, and let $D$ be a planar domain consisting of the points on one side of $\gamma$, within a suitable distance $\delta$ of $\gamma$. Denote by $\mu_1^{odd}(D)$ the smallest nontrivial Neumann eigenvalue having a corresponding eigenfunction that is odd with respect to the $y$-axis.  If $\gamma$ satisfies some simple geometric conditions,  then $\mu_1^{odd}(D)$ can be sharply estimated from below in terms of the length of $\gamma$, its curvature, and $\delta$.  Moreover, we give explicit conditions on $\delta$ that ensure $\mu_1^{odd}(D)=\mu_1(D)$.  Finally, we can extend our bound on $\mu_1^{odd}(D)$ to a certain class of three-dimensional domains.  In both the two- and three-dimensional settings, our domains are generically non-convex.
\end{abstract}

\vskip 1cm
\section{Introduction}\label{Sec: Intro}

Let $D\subset \R^n$ be a bounded,  connected, Lipschitz domain. We study the classical free membrane problem in $D$, that is,
\begin{equation}\label{NP}
\left\{
\begin{array}{ll}
-\Delta u=\mu u &\mbox{in}\>D
\\ \\
\frac{\partial u}{\partial \mathbf{n}}=0 &\mbox{on}\> \partial D,
\end{array}
\right.
\end{equation}
where $\mathbf{n}$ denotes the exterior unit normal to $\partial D$. We
arrange the eigenvalues of \eqref{NP} in a non-decreasing sequence $\{\mu_n(D)\}_{n \in \N_0}$, where each eigenvalue is 
repeated according to its multiplicity. The first eigenfunction of \eqref{NP} is clearly a constant with
eigenvalue $\mu_0(D)=0$ for any $D$. We shall be interested in the first non-trivial eigenvalue $\mu_1(D)$, which admits the following variational characterization:
$$
\mu_1(D)=\min\left\{\frac{\dint_D|\nabla \psi|^2 dxdy}{\dint_D \psi^2 dxdy}: \> \psi \in H^1(D)\setminus\{0\},\> \int_D\psi dxdy=0\right\},
$$
where $H^1(D)$ is the usual Sobolev space of square-integrable functions with weak first-order partials that are also square-integrable; all functions considered here and in what follows are real-valued.

As is well known, many difficulties arise in estimating $\mu_1(D)$. One reason for this is the lack of monotonicity of eigenvalues with respect to set inclusion. Another is the fact that eigenfunctions corresponding to $\mu_1(D)$ must change sign, and localizing the nodal line seems to be a hard problem (e.g., \cite{J}). 

Despite these difficulties, there are lower bounds on $\mu_1(D)$ in certain situations.  The celebrated Payne-Weinberger \cite{PW} inequality states that if $D$ is a convex domain with diameter $d(D)$, then
\begin{equation}\label{pworig}
\mu_1(D)\ge \frac{\pi^2}{d(D)^2}.
\end{equation}
The above estimate is asymptotically sharp, since $\mu_1(D)d(D)^2$ tends to $\pi^2$ for a parallelepiped all but one of whose dimensions shrink to 0. Estimate \eqref{pworig} fails for general non-convex sets, as can be seen by considering a domain consisting of two identical squares connected by a thin corridor. Such a counterexample suggests that a lower bound on $\mu_1(D)$ for non-convex domains should involve geometric quantities other than the diameter. In \cite{BCT_CPDE, BCT} such a lower bound involves the isoperimetric constant relative to $D$, and in \cite{GU} a lower bound is given in terms of an $L^\alpha$ norm of the Riemann conformal mapping of the unit disk onto $D$. Thus the problem of finding a lower bound on $\mu_1(D)$ for non-convex domains is often shifted to another geometric problem.  Related and further results may be found, for instance, in \cite{BCT_DIE,BT,EFK,EHL,FNT, V}.

We consider a class of domains that have a line or plane of symmetry, but that are typically non-convex.  Letting $\mu_1^{odd}$ denote the smallest nontrivial Neumann eigenvalue having a corresponding eigenfunction that is odd with respect to this line or plane, we give explicit lower bounds on $\mu_1^{odd}$. In the two-dimensional case, we let $\gamma(s)=(x(s),y(s)),\> s \in [0, L],$ 
be a smooth, non-closed, simple curve, parametrized with respect to its arc length, and whose image is symmetric with respect to the $y$-axis. That is,
$$
x(L-s)=-x(s), \quad y(L-s)=y(s), \quad s \in \left[0,\frac{L}{2}\right].
$$
Consider the domain $D$ consisting of the points on one side of $\gamma$, within a suitable distance $\delta$ of $\gamma$.  Using the normal vector to $\gamma(s)$ obtained by rotating $\gamma'(s)$ clockwise by $\frac{\pi}{2}$, we may describe $D$ as follows (see Figure 1):  
\begin{equation} \label{D}
D=\left\{(x(s)+ry'(s), y(s)-rx'(s)): \> s \in (0,L), \> r \in (0,\delta)\right\}.
\end{equation}

 \begin{figure}[ht]\label{fig}
\centering
 \includegraphics[scale=0.3]{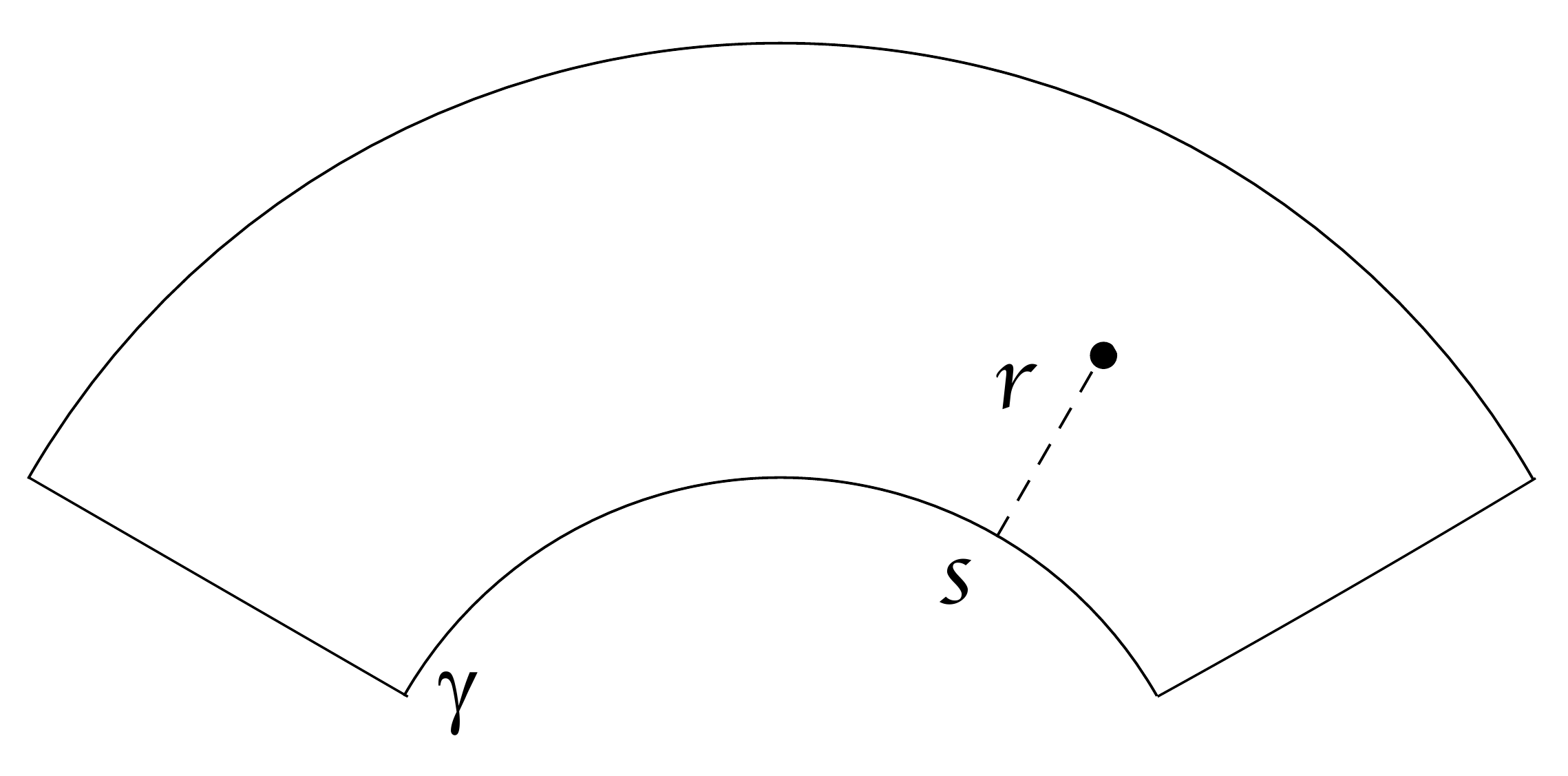}
\caption{A typical domain $D$ described in the Fermi coordinate system.}
\end{figure}

Denote by $\mu_1^{odd}$ the smallest nontrivial Neumann eigenvalue having a corresponding eigenfunction that is odd with respect to the $y$-axis. Our main result is
\begin{theorem}\label{1}
Suppose that the curvature $k(s)$ of $\gamma$ is concave in $[0,L]$ and let $\delta>0$ be such that 
$1+\delta k(s) > 0$ in $[0,L]$.  If $D$ is simply connected, then
\begin{equation*}
\mu_1^{odd}(D) \ge B \frac{\pi^2}{L^2},
\end{equation*}
where $B= \displaystyle{\min_{r\in \left[0,\delta\right], \, s\in \left[0,L \right]}}\frac{1}{(1+rk(s))^2}$; equality holds if $\gamma$ is a line segment.
\end{theorem}
Thus we give a sharp lower bound on $\mu_1^{odd}$ that is reminiscent of the bound in \cite{PW}, with a correction factor that encodes the relevant geometry of our domains.  We stress that this result falls in the category of lower bounds obtained in \cite{PW,BCT_CPDE,BCT,GU}, since under certain explicit assumptions on $L$ and $\delta$, we show that $\mu_1(D)$ coincides with $\mu_1^{odd}(D)$ (see Propositions \ref{2d} and \ref{2d-1}). Roughly speaking, this phenomenon occurs whenever $\delta$ is sufficiently smaller than $L$. If such a relationship does not hold, Theorem  \ref{1} is still relevant, as $\mu_1^{odd}(D)$ can be realized as the lowest eigenvalue of the Laplacian with mixed boundary conditions on $D^{+}=\{(x,y)\in D:x>0\}$. Sharp bounds for such eigenvalues have been obtained in \cite{AC} (see also \cite[\S 2.5]{B} and \cite{LP}).  Finally, we are able to adapt the argument used to prove Theorem \ref{1} to give the same lower bound on $\mu_1^{odd}$ for certain three-dimensional domains that are not necessarily convex.

The paper is organized as follows.  In \S 2, we prove Theorem \ref{1}.  In \S 3, we give conditions under which $\mu_1^{odd}(D)$ coincides with $\mu_1(D)$, as well as examples illustrating our two-dimensional results.  We extend our two-dimensional results to certain three-dimensional domains in \S 4, and conclude with an appendix that details some of the computations associated with the Fermi coordinate system that we use in our proofs. 

%--------------------------------------------------------------------------------------------------------------------------------------------------------------------------------------------------------------------------------------------------------------------------------------------------------------------------------------------------------------------------------------------------------------

\section{Proof of the main result}

The focus of this section will be the proof of Theorem \ref{1}.  We will introduce a Fermi coordinate system on $D$ and slice $D$ into thin pieces, with the mean value of an odd eigenfunction vanishing on each slice.  Since the slices are thin, we are close to being in a one-dimensional setting and the following lemma from \cite[\S 2]{PW} will play a key role.

\begin{lemma}\label{2}
Let $p(s)$ be a concave, non-negative function on the interval $[0,L]$. Then for any piecewise twice differentiable function $v(s)$ that satisfies
$$
\int_0^L v(s)p(s)ds=0,
$$
it follows that
$$
\int_0^L (v'(s))^2 p(s)ds \ge \frac{\pi^2}{L^2}\int_0^L v(s)^2 p(s)ds.
$$
\end{lemma}

\begin{remark}\label{obs}
Suppose $p$ is also even with respect to $\frac{L}{2}$, and $v(s)$ is a sufficiently smooth function satisfying $v\left(\dfrac L 2\right)=0$.  Define a new function $w(s)$ that is equal to $v(s)$ on $[0, \frac{L}{2}]$, and is equal to the odd reflection of $v(s)$ in the line $s = \frac{L}{2}$ on $[\frac{L}{2},L]$.  Then Lemma \ref{2} applies to $w(s)$ and a straightforward computation shows that
\[
\int_0^{\frac{L}{2}}(w'(s))^2 p(s) ds \geq \frac{\pi^2}{L^2} \int_0^{\frac{L}{2}} w(s)^2 p(s) ds.
\]  
Since $w(s) = v(s)$ on $[0, \frac{L}{2}]$, we may replace $w$ by $v$ in the preceding inequality.  We will use this observation in our proof of Proposition \ref{2d-1}.
\end{remark}

We are now ready to prove Theorem \ref{1}.

\begin{proof}[Proof of Theorem \ref{1}]
Fix $n \in \N$. Let us denote by $\dist_\gamma(x,y)$ the distance of a generic point $(x,y)\in D$ to $\gamma$ and, for any $i=0,\dots,n-1$, by
$$
D_i=\left\{(x,y)\in D:\> \frac{i\,\delta}{n}<\dist_\gamma(x,y)<\frac{(i+1)\,\delta}{n}\right\}.
$$

\noindent Let $u$ be an odd eigenfunction corresponding to $\mu_1^{odd}(D)$ (from now on, for the sake of brevity, we will
omit ``with respect to the $y$-axis"). Using the definition of eigenfunction and a Green's formula, we see that
$$
\mu_1^{odd}(D)=\frac{\dint_D |\nabla u|^2 dxdy}{\dint_D u^2 dxdy};
$$
moreover, the fact that $u$ is odd implies
$$ 
\int_{D_i} u\,dxdy=0 \quad \text{for all } i=0,\dots,n-1.
$$

\noindent We want to evaluate the energy of $u$ in any $D_i$. We construct a Fermi coordinate system $(r,s)$ whereby points $(x,y)$ in $D$ are determined by specifying the distance $r=\dist_\gamma(x,y)$ to the curve $\gamma$, and the arc length $s$ of the point on $\gamma$ nearest to $(x,y)$.  Alternatively, we observe that the co-area formula on the level sets of the distance to $\gamma$ yields the same results. Changing from rectangular to Fermi coordinates (see the Appendix for details), we have
\begin{eqnarray*}
\int_{D_i} |\nabla u|^2 dxdy &=& \int_0^L\left(\int_{\frac{i\,\delta}{n}}^{\frac{(i+1)\,\delta}{n}}\left(\frac{1}{1+rk(s)}u_s^2(r,s)+(1+rk(s))u_r^2(r,s)\right)dr\right)ds
\\
&\ge& \min_{s\in [0,L], \,r\in \left[\frac{i\,\delta}{n},\frac{(i+1)\,\delta}{n}\right]
}\frac{1}{(1+rk(s))^2}\int_0^L\left(\int_{\frac{i\,\delta}{n}}^{\frac{(i+1)\,\delta}{n}}u_s^2(r,s)(1+rk(s))\,dr\right)ds
\\
&\ge&B\int_0^L\left( \int_{\frac{i\,\delta}{n}}^{\frac{(i+1)\,\delta}{n}}u_s^2(r,s)(1+rk(s))\,dr\right)ds ,
\end{eqnarray*}
where the first inequality follows from the hypothesis that $1+\delta k(s)>0$ in $[0,L]$, and the second from the definition of $B$. 
Let us write
$$
\int_0^L\left( \int_{\frac{i\,\delta}{n}}^{\frac{(i+1)\,\delta}{n}}u_s^2(r,s)(1+rk(s))\,dr\right)ds=I_1+I_2,
$$
where
\begin{equation*}
I_1=\int_0^L\left(\int_{\frac{i\,\delta}{n}}^{\frac{(i+1)\,\delta}{n}}\left(u_s^2(r,s)-u_s^2\left(\frac{i\,\delta}{n},s\right)\right)(1+rk(s))dr\right)ds
\end{equation*}
and
\begin{equation*}
I_2=\int_0^L\left(\int_{\frac{i\,\delta}{n}}^{\frac{(i+1)\,\delta}{n}}u_s^2\left(\frac{i\,\delta}{n},s\right)(1+rk(s))dr\right)ds
=\frac{\delta}{n}\int_0^Lu_s^2\left(\frac{i\,\delta}{n},s\right)\left(1+\frac{1+2i}{2} \frac{\delta}{n}k(s) \right)ds.
\end{equation*}
Let $A$ be a common bound for the absolute value of each of $u$ and its first and second derivatives when expressed in Fermi coordinates. Applying the Mean Value Theorem, we deduce that
\begin{equation}\label{i1}
|I_1| \le \frac{2A^2\delta}{n}|D_i|.
\end{equation}

\noindent We will return to $I_2$ in a moment; first, we note that the arguments used above may be applied to show that
$$
\int_{D_i} u^2dxdy=\int_0^L\left(\int_{\frac{i\,\delta}{n}}^{\frac{(i+1)\,\delta}{n}}u^2(r,s)(1+rk(s))\,dr\right)ds=J_1+J_2,
$$
where
\begin{equation*}
J_1=\int_0^L\left(\int_{\frac{i\,\delta}{n}}^{\frac{(i+1)\,\delta}{n}}\left(u^2(r,s)-u^2\left(\frac{i\,\delta}{n},s\right)\right)(1+rk(s))dr\right)ds ,
\end{equation*}
\begin{equation*}
J_2=\int_0^L\left(\int_{\frac{i\,\delta}{n}}^{\frac{(i+1)\,\delta}{n}}u^2\left(\frac{i\,\delta}{n},s\right)(1+rk(s))dr\right)ds
=\frac{\delta}{n}\int_0^Lu^2\left(\frac{i\,\delta}{n},s\right)\left(1+\frac{1+2i}{2} \frac \delta n k(s)\right)ds,
\end{equation*}
and
$$
|J_1|\le \frac{2A^2\delta}{n}|D_i|.
$$

\noindent We will next relate $I_2$ and $J_2$ via Lemma \ref{2}. Using the expression for signed curvature that may be found in the Appendix, it is straightforward to show that $k(s)$ is even with respect to $\frac{L}{2}$.  Since $u$ is odd with respect to $\frac{L}{2}$, we have that
$$
\int_0^Lu\left(\frac{i\,\delta}{n},s\right)\left(1+\frac{1+2i}{2} \frac \delta n k(s) \right)ds=0.
$$
Our hypothesis that $1 + \delta k(s) > 0$ for $s$ in $[0,L]$ implies that $1+\frac{1+2i}{2} \frac \delta n k(s) >0$.  Since we have also assumed that $k(s)$ is concave in $[0,L]$, Lemma \ref{2} implies that
\begin{equation}\label{i2}
I_2 \ge \frac{\pi^2}{L^2}J_2.
\end{equation}

\noindent We now combine the above estimates.  We have
\begin{eqnarray*}
\int_{D_i}|\nabla u|^2dxdy &\ge&B\left(I_1+I_2\right)
\\
&\ge&B\frac{\delta}{n}\left(\int_0^Lu_s^2\left(\frac{i\,\delta}{n},s\right)\left(1+\frac{1+2i}{2} \frac{\delta}{n} k(s) \right)ds-2A^2|D_i|\right)
\\
&\ge& B\frac{\delta}{n}\left( \frac{\pi^2}{L^2}\int_0^Lu^2\left(\frac{i\,\delta}{n},s\right)\left(1+\frac{1+2i}{2} \frac{\delta}{n} k(s) \right)ds-2A^2|D_i|\right),
\end{eqnarray*}
where we used \eqref{i1} and \eqref{i2}, respectively.  Using an equivalent expression for $J_2$, converting back to rectangular coordinates, and subtracting a positive term, we conclude that
\begin{equation*}
\int_{D_i}|\nabla u|^2dxdy \ge B \frac{\pi^2}{L^2}\left(\int_{D_i}u^2\, dxdy-\frac{2A^2\delta}{n}|D_i|\right)-B\frac{2A^2\delta}{n}|D_i|.
\end{equation*}
Summing over $i$, we obtain
$$
\int_D |\nabla u|^2 dxdy \ge B \frac{\pi^2}{L^2} \int_D u^2 dxdy -\frac{C\delta}{n}|D|,
$$
with $C>0$. Taking the limit as $n$ goes to $+\infty$ yields the claim.

Finally, to establish the case of equality, we take $\gamma(s)=\left(\dfrac{L}{2}-s,0\right)$ with $s\in [0,L]$ so that $D=\left(-\dfrac{L}{2},\dfrac{L}{2}\right)\times(0,\delta)$. In this case,   $\mu_1^{odd}(D)=\dfrac{\pi^2}{L^2}$.
\end{proof}

\begin{remark}
If $\gamma$ is a curve as in Theorem \ref{1}, but is closed, it follows from the Four Vertex Theorem that $\gamma$ is a circle and $D$ is an annulus; the eigenvalues of such domains may be found exactly from equations that involve cross products of derivatives of Bessel functions.
\end{remark}

\begin{remark}
If $\gamma$ is part of the boundary of a convex domain $T$, so that $k(s)\ge 0$ in $[0,L]$, then $1+\delta k(s)$ is clearly positive for any choice of  $\delta>0$.  Thus one may remove the restriction on the value of $\delta$ from Theorem \ref{1} for such a $\gamma$. 
\end{remark}

\noindent Some concrete examples to which Theorem \ref{1} applies will be provided at the end of \S 3.

%--------------------------------------------------------------------------------------------------------------------------------------------------------------------------------------------------------------------------------------------------------------------------------------------------------------------------------------------------------------------------------------------------------------

\section{A sufficient condition for $\mu_1(D)=\mu_1^{odd}(D)$}

In this section we give some geometric conditions to ensure that $\mu_1(D)$ coincides with $\mu_1^{odd}(D)$.  Arguments of a similar flavor have been used in \cite{BB,J}.

\begin{proposition}\label{2d}
 Let $\gamma$ and $D$ be as in Theorem \ref{1} and suppose that $\gamma$ may be realized as the graph of a function.  We denote by $\Pi_x(D)=(-P,P)$ the projection of $D$ onto the $x$-axis.  Let $S_{x}$ denote the vertical cross sections of $D$, i.e., $S_{x}=\left\{ (\widetilde{x},\widetilde{y})\in D :\widetilde{x}=x\right\}$, and define $S=\max\limits_{x\in \left[ 0,P\right) }\left\vert
S_{x}\right\vert$.  If
\begin{equation}\label{so}
S^2<P^2\frac{\dint_{D }\sin ^{2}\left( \frac{
\pi }{2P}x\right) dxdy}{\dint_{D }\cos ^{2}\left( \frac{\pi }{2P}x
\right) dxdy},
\end{equation}
then
$$\mu_1(D)=\mu _{1}^{odd}(D).$$ 
\end{proposition}

\begin{proof} Suppose for the sake of reaching a contradiction that there is no odd eigenfunction
corresponding to $\mu _{1}(D)$. Therefore if $v(x,y)$ is any
eigenfunction corresponding to $\mu _{1}(D)$, then $u(x,y)=v(x,y)+v(-x,y)$ is an eigenfunction that is even. 

We begin by showing that the curve $\gamma^\delta$ parallel to $\gamma$ at distance $\delta$ must also be the graph of a function.  Note that $\gamma(s)$ restricted to either $\left[0, \dfrac{L}{2}\right]$ or $\left[\dfrac{L}{2}, L\right]$ lies in the first quadrant; we assume that $\left[0, \dfrac{L}{2}\right]$ is the relevant interval.  Thus $\gamma(s)$, for $0 \leq s \leq \dfrac{L}{2}$, is the graph of a function in the first quadrant and may be parametrized by $\psi(t) = (T-t, f(T-t))$ for $0 \leq t \leq T$ and some function $f$.  Our parametrization is constructed so that we traverse $\gamma$ with its original orientation.  The curve 
\[
\gamma^\delta(s) = (x(s) + \delta y'(s), y(s) - \delta x'(s)), \text{ for } 0 \leq s \leq \frac{L}{2},
\]
may then be parametrized by
\[
\psi^\delta (t) = \left(T-t-\delta \frac{f'(T-t)}{\sqrt{1+(f'(T-t))^2}}\, , f(T-t) + \delta \frac{1}{\sqrt{1+(f'(T-t))^2}}\right), \ \ 0 \leq t \leq T,
\]
where, as we did for $\gamma^\delta(s)$, we are again translating along a normal vector obtained by rotating our original tangent vector clockwise by $\dfrac{\pi}{2}$.
Taking the derivative with respect to $t$ of the first coordinate of $\psi^\delta(t)$, we find that it equals $-1-\delta k(t)$, where $k(t)$ is the signed curvature of $\psi(t)$.  However, we parametrized $\psi(t)$ so that it would have the same orientation as $\gamma(s)$, so our assumption that $1 + \delta k(s) >0$ implies that $1 + \delta k(t)>0$. Hence the first coordinate of $\psi^\delta(t)$ is strictly decreasing, and we deduce that $\psi^\delta(t)$ is also the graph of a function.

Next we use nodal considerations to restrict our attention to a subset of $D$.  As is well-known (e.g., \cite{CF}), the nodal line $u=0$ is a smooth curve; moreover, it cannot enclose any subdomain of $D$.  Our assumption that $u$ is even implies that the nodal domains corresponding to $u$ are symmetric with respect to the $y$-axis, and Courant's theorem implies that there are exactly two such nodal domains. Thus the nodal line intersects $\partial D$ in exactly two symmetric points and it crosses the $y$-axis at precisely one point inside $D$.
Let $\Pi_x(\{u=0\})=[-Z,Z]$ be the projection of the nodal line onto the $x$-axis.  Since the nodal line is a smooth curve, we see that each vertical line $x=c$, where $-Z \leq c \leq Z$, intersects the nodal line.  Let
\[
D _{+}=\left\{ (x,y)\in D :u(x,y)>0\right\} \text{ and } D_{-} = \left\{ (x,y)\in D :u(x,y)<0\right\}.
\]
We claim that the projection of at least one of $D_{+}$ and $D_{-}$ onto the $x$-axis is contained in $[-Z,Z]$.  If this were not the case, we could find points $(x_1, y_1) \in D_{+}$ and $(x_2, y_2) \in D_{-}$ with $|x_1|, |x_2|>Z$.  Since $u$ is even, we may assume that $x_1, x_2 >Z$.  We claim that we may connect $(x_1, y_1)$ to $(x_2, y_2)$ via a path whose $x$-coordinate is always strictly larger than $Z$. 
Define 
\[
\gamma^r(s) = (x(s) + r y'(s), y(s) - r x'(s)), \text{ for } 0 \leq s \leq \frac{L}{2} \text{ and } 0 \leq r \leq \delta.
\]
Note that, for $i=1,2$, we have $(x_i, y_i)=\gamma^{r_i}(s_i)$ for some $r_i \in [0, \delta]$ and some $s_i \in [0,\dfrac{L}{2}]$.  Fixing $s$ and letting $r$ vary between $0$ and $\delta$, we see that $\gamma^r(s)$ traces out a line segment.  Thus we may travel along such line segments from $\gamma^{r_i}(s_i)$ to either $\gamma(s_i)$ or $\gamma^{\delta}(s_i)$ for $i=1,2$ in such a way that the $x$-coordinate remains strictly greater than $Z$.  Our path from $\gamma^{r_1}(s_1)$ to $\gamma^{r_2}(s_2)$ is then completed by traveling appropriately along the boundary; we know that the boundary portion of our path has $x$-coordinate strictly greater than $Z$ because $\gamma$ and $\gamma^{\delta}$ are both graphs of functions, and $\gamma^r(0)$ is a line segment.  By the Intermediate Value Theorem, the nodal line intersects this path, which is a contradiction. Thus the projection of at least one of $D_{+}$ and $D_{-}$ onto the $x$-axis is contained in $[-Z,Z]$; replacing $u$ with $-u$ as needed, we may assume that the projection of $D_{+}$ is contained in $[-Z,Z]$. 

We will now use $D_{+}$ to find a lower bound on $\mu_1(D)$.  We have
\begin{equation}
\mu _{1}(D)=\frac{\dint_{D _{+}}\left\vert \nabla u\right\vert ^{2}dxdy}{
\dint_{D _{+}}u^{2}dxdy}\geq \frac{\dint_{D _{+}}u_y^2\,dxdy}{\dint_{D _{+}}u^{2}dxdy}=\frac{
\dint_{\Pi_x(D_{+})} \left( \dint_{S_{x}\cap D_+}u_y^2\,dy\right) dx}{\dint_{D _{+}}u^{2}dxdy}. \label{Fubini}
\end{equation}
For almost every $x$ we have 
\begin{equation*}
S_x\cap D_{+}=\bigcup_{j=1}^{\infty} I_j^x,
\end{equation*}
where for any $j$, $I_j^x$ is an open interval such that $u$ vanishes at one or both endpoints of $I_j^x$. The boundary condition is potentially unknown at one of the endpoints of $I_j^x$, but we may take an odd reflection of $u$ in the Dirichlet end of $I_j^x$.  Thus $u$ has mean value equal to zero on the doubled $I_j^x$, and we have

$$
\int_{I_x^j}u_y^2\, dy \ge \frac{\pi^2}{4|I_x^j|^2}\int_{I_x^j}u^2dy \ge \frac{\pi^2}{4S^2}\int_{I_x^j}u^2dy.
$$
This last consideration, together with (\ref{Fubini}), yields
\[
\mu _{1}(D)\geq \frac{\pi ^{2}}{4S^2}.
\]
On the other hand, choosing $\sin \left( \frac{\pi }{2P}x\right) $ as a test function for $\mu _{1}(D)$ we obtain
\[
\mu _{1}(D)\leq \frac{\pi ^{2}}{4P^{2}}\frac{\dint_{D }\cos ^{2}\left( 
\frac{\pi }{2P}x\right) dxdy}{\dint_{D }\sin ^{2}\left( \frac{\pi }{2P}x
\right) dxdy};
\]
combining these two inequalities on $\mu_1(D)$, we see that we have a contradiction to our hypothesis \eqref{so}. 
\end{proof}

\begin{remark} Proposition \ref{2d} can be stated in different ways depending
on the choice of the test function used to obtain the upper bound for $\mu
_{1}(D)$. A rough estimate can be obtained by choosing $x$ as a test function. In this case condition \eqref{so} becomes
\begin{equation*}
S^{2}<\frac{\pi ^{2}}{4}\dfrac{\dint_{D }x^{2}dxdy}{\left\vert D
\right\vert }.  
\end{equation*}
Since our domain $D$ has a special shape, we can alternatively use $\cos\left(\frac{\pi}{L}s\right)$ as a test function in the Rayleigh quotient written in Fermi coordinates and \eqref{so} becomes
\begin{equation}\label{cos}
S^2<\frac{L^2}{4}\frac{\dint_0^L\left(\dint_0^\delta \cos^2\left(\frac{\pi}{L} s\right)(1+rk(s))dr\right)ds}{\dint_0^L\left(\dint_0^\delta \sin^2\left(\frac{\pi}{L} s\right)\frac{1}{1+rk(s)}dr\right)ds}.
\end{equation}
 \end{remark}
 
 \bigskip

In the next proposition we show that, if $\gamma$ is not the graph of a one-dimensional function, it is still possible to give a condition ensuring that $\mu_1(D)=\mu_1^{odd}(D)$. 

\begin{proposition}\label{2d-1}
 Let $\gamma$ and $D$ be as in Theorem \ref{1}. Then
$$\mu_1(D)=\mu _{1}^{odd}(D)$$
if one of the following alternatives holds:
 \begin{enumerate}
 \item  $k(s)\ge0$ for all $s \in [0,L]$ and 
 \begin{equation}\label{delta}
 \underset{s\in[0,L]}{\max}\delta^2 (2+\delta k(s))^2<  \dfrac{L^2}{\pi^2}\> \frac{\dint_0^L\left(\dint_0^\delta \cos^2\left(\frac{\pi}{L} s\right)(1+rk(s))dr\right)ds}{\dint_0^L\left(\dint_0^\delta \sin^2\left(\frac{\pi}{L} s\right)\frac{1}{1+rk(s)}dr\right)ds};
\end{equation}

\item $k(s)<0$ for all $s \in [0,L]$ and
\begin{equation}\label{delta1}
 \underset{s\in[0,L]}{\max}\dfrac{4\delta^2}{ (1+\delta k(s))^2}<  \dfrac{L^2}{\pi^2}\> \frac{\dint_0^L\left(\dint_0^\delta \cos^2\left(\frac{\pi}{L} s\right)(1+rk(s))dr\right)ds}{\dint_0^L\left(\dint_0^\delta \sin^2\left(\frac{\pi}{L} s\right) \frac{1}{1+rk(s)}dr\right)ds}; or
\end{equation}

\item $k(s)$ changes its sign in $[0,L]$, and
\begin{equation}\label{delta2}
\max\left\{  \underset{s\in[0,L]}{\max}\delta^2 (2+\delta k(s))^2,  \underset{s\in[0,L]}{\max}\dfrac{4\delta^2}{ (1+\delta k(s))^2}\right\}<  \dfrac{L^2}{\pi^2}\> \frac{\dint_0^L\left(\dint_0^\delta \cos^2\left(\frac{\pi}{L} s\right)(1+rk(s))dr\right)ds}{\dint_0^L\left(\dint_0^\delta \sin^2\left(\frac{\pi}{L} s\right)\frac{1}{1+rk(s)}dr\right)ds}.
\end{equation}
 \end{enumerate}
  \end{proposition}

\begin{proof}
As in the proof of Proposition \ref{2d}, suppose for the sake of reaching a contradiction that there is no odd eigenfunction
corresponding to $\mu _{1}(D)$. Therefore if $v(x,y)$ is any
eigenfunction corresponding to $\mu _{1}(D)$, then $u(x,y)=v(x,y)+v(-x,y)$ is an eigenfunction that is even. 

Denote $\partial D=\gamma \cup \gamma^\delta \cup S$, where $S$ is the union of the two segments joining $\gamma$ and $\gamma^\delta$, and let $\{P_L,P_R\}=\partial D \cap \{u=0\}$. Of course, $P_L$ and $P_R$ are symmetric points with respect to the $y$-axis.
Exactly one of the following cases occurs:
\begin{itemize}
\item[$(i)$] $P_L,P_R \in \gamma$;
\item[$(ii)$] $P_L,P_R \in S$; or
\item[$(iii)$] $P_L,P_R\in \gamma^\delta$.
\end{itemize}

We begin by treating case (1) in the statement of Proposition \ref{2d-1}; we will analyze subcase $(i)$ first, and then handle subcases $(ii)$ and $(iii)$ together.
We denote by $\lambda^{ND}(D)$ the lowest eigenvalue of the following mixed Dirichlet-Neumann problem:
\begin{equation}\label{ND}
\left\{
\begin{array}{ll}
-\Delta \psi = \lambda \psi & \mbox{in}\> D
\\ \\
\frac{\partial \psi}{\partial \mathbf{n}}=0 &\mbox{on} \>\wideparen{P_LP_R}
\\ \\
\psi=0 & \mbox{on} \>\partial D\setminus \wideparen{P_LP_R}
\end{array}
\right.
\end{equation}
where $\wideparen{P_LP_R}$ is the connected portion of $\gamma$ with endpoints $P_L$ and $P_R$. Without loss of generality we may assume that $u>0$ in $D_+$, where $\partial D_+\cap \gamma =\wideparen{P_LP_R}$. Let $u_+$ denote the positive part of $u$. Using $u_+$ as a test function in the variational characterization of $\lambda^{ND}(D)$, we obtain
\begin{equation}\label{bb}
\mu_1(D)=\dfrac{\int_{D_+} |\nabla u|^2 dxdy}{\int_{D_+} u^2 dxdy}=\dfrac{\int_{D} |\nabla  u_+|^2 dxdy}{\int_{D}  u_+^2 dxdy}\ge \lambda^{ND}(D)=\dfrac{\int_{D} |\nabla \psi|^2 dxdy}{\int_{D} \psi^2 dxdy},
\end{equation}
where $\psi$ is an eigenfunction of problem \eqref{ND} corresponding to $\lambda^{ND}(D).$
By using a Fermi coordinate system  we can estimate the last term in \eqref{bb}, obtaining
\begin{equation}\label{Fubini1}
\mu _{1}(D)\geq \frac{\dint_0^L \left(\dint_0^\delta \psi_r^2 (1+rk(s))dr\right)ds}{\dint_0^L \left(\dint_0^\delta \psi^2 (1+rk(s))dr\right)ds}.
\end{equation}
Note that if $r=\delta$, then $\psi=0$ for any $s\in [0,L]$. To estimate the integral $\dint_0^{\delta}\psi_r^2(1+rk(s))dr$, we consider the odd and even extensions (with respect to $\delta$) of $\psi$ and $1+rk(s)$ to $[0,2\delta]$, respectively. Since $k(s)\ge 0$, the latter extension is concave in $r$. Hence Remark \ref{obs} implies 
$$
\int_0^{\delta}\psi_r^2(1+rk(s))dr\geq \frac{\pi^2}{4\delta^2}\int_0^{\delta}\psi^2(1+rk(s))dr.
$$
Integrating with respect to $s$ gives
\begin{equation*}
\dint_0^L\left(\dint_0^\delta \psi_r^2 (1+rk(s))dr\right)ds \ge  \dfrac{\pi^2}{4\delta^2}\dint_0^L\left(\dint_0^\delta \psi^2 (1+rk(s))dr\right)ds,
\end{equation*}
and combining this inequality with \eqref{Fubini1} yields
\begin{equation}\label{lb1}
\mu_1(D) \ge \frac{\pi^2}{4\delta^2} \ge \dfrac{1}{\underset{s\in [0,L]}{\max} \delta^2(2+\delta k(s))^2},
\end{equation}
with the last inequality holding by the non-negativity assumption on $k$. 
On the other hand, choosing $\cos\left(\dfrac{\pi}{L}s\right)$ as test function in the variational characterization of $\mu_1(D)$ where the Rayleigh quotient is written in Fermi coordinates, we obtain
$$
\mu_1(D) \le \frac{\pi^2}{L^2} \frac{\dint_0^L\left(\dint_0^\delta \sin^2\left(\frac{\pi}{L} s\right)\frac{1}{1+rk(s)}dr\right)ds}{\dint_0^L\left(\dint_0^\delta \cos^2\left(\frac{\pi}{L} s\right)(1+rk(s))dr\right)ds},
$$
reaching a contradiction.

In subcase $(ii)$, define $\wideparen{P_LP_R}$ to be the path on $\partial D$ connecting $P_L$ and $P_R$ with nonempty intersection with $\gamma^\delta$; in subcase $(iii)$, define $\wideparen{P_LP_R}$ to be the path on $\partial D$ connecting $P_L$ and $P_R$ that has empty intersection with $S$. We denote by $\lambda^{ND}(D)$ the lowest eigenvalue of the mixed Dirichlet-Neumann problem given by \eqref{ND}. Without loss of generality we may assume that $\partial D_+\cap (\gamma^\delta\cup S)=\wideparen{P_LP_R}$. We proceed as in subcase $(i)$ through \eqref{Fubini1}. Note that if $r=0$, then $\psi=0$ for any $s \in [0,L]$ since $\wideparen{P_LP_R} \cap \gamma = \emptyset$.  Suppose $k(s)=0$ so that we wish to estimate $\dint_0^{\delta}\psi_r^2dr$; we consider the odd extension (with respect to 0) of $\psi$ to $[-\delta,\delta]$. Then Remark \ref{obs} implies 
\begin{equation}\label{iii}
\int_0^{\delta}\psi_r^2dr\geq \frac{\pi^2}{4\delta^2}\int_0^{\delta}\psi^2dr.
\end{equation}
Suppose $k(s)>0$ and define
\begin{eqnarray}\label{b1squared}
B_1^2(s)&=&\max_{r \in [0,\delta] }\left(\int_r^\delta (1+tk(s))dt\right)\left(\int_0^r \dfrac{1}{1+tk(s)}dt\right)
\\
&=&
\displaystyle\max_{r \in [0,\delta] } (\delta -r)\left(1+(\delta+r)\dfrac{k(s)}{2}\right)\dfrac{\log(1+rk(s))}{k(s)}. \nonumber
\end{eqnarray}
By \cite[p. 40, Thm. 1]{M}, we have 
\begin{equation}\label{bar}
\dint_0^\delta \psi_r^2 (1+rk(s))dr \ge C_1(s) \dint_0^\delta \psi^2 (1+rk(s))dr ,
\end{equation}
where 
$$ C_1(s) \ge \dfrac{1}{4B_1^2(s)}.$$
Set
$$b_1(r)=(\delta -r)\left(1+(\delta+r)\dfrac{k(s)}{2}\right)\dfrac{\log(1+rk(s))}{k(s)}.$$
Then $b_1(0)=b_1(\delta)=0$ and
$$
b_1'(r)=-\left(1+(\delta+r)\dfrac{k(s)}{2}\right)\dfrac{\log(1+rk(s))}{k(s)}+(\delta-r)\dfrac{k(s)}{2}\dfrac{\log(1+rk(s))}{k(s)}+(\delta - r)\left(1+(\delta+r)\dfrac{k(s)}{2}\right)\dfrac{1}{1+rk(s)}.
$$
Thus, if $\bar r\in (0,\delta)$ is a maximum point for $b_1$, so that $b_1'(\bar r)=0$, then we have
$$
\dfrac{\log(1+\bar r k(s))}{k(s)}=(\delta - \bar r)\left(1+(\delta+\bar r)\dfrac{k(s)}{2}\right)\dfrac{1}{(1+\bar r k(s))^2}.
$$
This implies that
\begin{eqnarray}
 B_1^2(s)&=&\max_{r \in [0,\delta]}b_1(r)=(\delta - \bar r)^2\left(1+(\delta +\bar r)\dfrac{k(s)}{2}\right)^2 \dfrac{1}{(1+\bar r k(s))^2} \label{b1}
  \\
 &\le & \delta^2 \left(1+\delta\dfrac{k(s)}{2}\right)^2, \label{b11} \notag
\end{eqnarray}
and hence
\begin{equation}\label{cs}
C_1(s) \ge \dfrac{1}{\delta^2(2+\delta k(s))^2}.
\end{equation}

Using \eqref{iii},\eqref{bar}, and \eqref{cs}, we deduce that
$$
\dint_0^\delta \psi_r^2 (1+r k(s))dr \ge  \left\{\begin{array}{ll} \dfrac{\pi^2}{ 4\delta^2}\dint_0^\delta \psi^2 (1+rk(s))dr & \mbox{if}\> k(s)=0, \\ \\ \dfrac{1}{\delta^2(2+ \delta k(s))^2}\dint_0^\delta \psi^2 (1+ r k(s))dr & \mbox{if}\> k(s)>0, \end{array}\right. 
$$
and therefore
$$
\dint_0^\delta \psi_r^2 (1+r k(s))dr \geq \frac{1}{\delta^2(2+\delta k(s))^2} \dint_0^\delta \psi^2 (1+rk(s))dr.
$$
Combining this inequality with \eqref{Fubini1} yields
\begin{equation}\label{lb2}
\mu_1(D) \geq \dfrac{1}{\underset{s\in [0,L]}{\max} \delta^2(2+\delta k(s))^2}.
\end{equation}
We conclude as in subcase $(i)$, thus completing case (1).

Next we treat case (2).  In all three subcases, we proceed as in case (1) through \eqref{Fubini1}. In order to estimate from below the ratio on the right-hand side in \eqref{Fubini1}, we  will again use \cite[p. 40, Thm. 1]{M}.  In subcase $(i)$, we take $\mu=\nu=(1 + (\delta-x)k(s))\textbf{1}_{[0,\delta]}(x)$ in that statement.  Denoting
\begin{eqnarray*}
B_2^2(s) &=&\max_{r \in [0,\delta] }\left(\int_0^{\delta-r} (1+t k(s))dt\right)\left(\int_{\delta-r}^\delta \dfrac{1}{1+t k(s)}dt\right)
\\
&=&\max_{r \in [0,\delta] } (\delta -r)\left(1+(\delta-r)\dfrac{k(s)}{2}\right)\dfrac{1}{(-k(s))}\log\left(\dfrac{1+(\delta-r)k(s)}{1+\delta k(s)}\right),
\end{eqnarray*}
we have
\begin{equation}\label{bar2}
\dint_0^\delta \psi_r^2 (1+r k(s))dr \ge C_2(s) \dint_0^\delta \psi^2 (1+r k(s))dr,
\end{equation}
where 
$$ C_2(s) \ge \dfrac{1}{4B_2^2(s)}.$$
Set
$$
b_2(r)=(\delta -r)\left(1+(\delta-r)\dfrac{k(s)}{2}\right)\dfrac{1}{(-k(s))}\log\left(\dfrac{1+(\delta-r)k(s)}{1+\delta k(s)}\right).
$$
Then $b_2(0)=b_2(\delta)=0$ and
\begin{eqnarray*}
b_2'(r)&=&-\left(1+(\delta-r)\dfrac{k(s)}{2}\right)\dfrac{1}{(-k(s))}\log\left(\dfrac{1+(\delta-r)k(s)}{1+\delta k(s)}\right)
\\
&&+\dfrac{(\delta-r)}{2}\log\left(\dfrac{1+(\delta-r)k(s)}{1+\delta k(s)}\right)+(\delta -r)\left(1+(\delta-r)\dfrac{k(s)}{2}\right)\dfrac{1}{1+(\delta-r)k(s)}.
\end{eqnarray*}
Thus, if $\bar r \in (0,L)$ is a maximum point for $b_2$, so that $b_2'(\bar r)=0$, then we have
$$
\dfrac{1}{(-k(s))}\log\left(\dfrac{1+(\delta-\bar{r})k(s)}{1+\delta k(s)}\right)=(\delta- \bar{r})\left(1+(\delta-\bar{r})\dfrac{k(s)}{2}\right)\dfrac{1}{(1+(\delta-\bar{r})k(s))^2}.
$$
This implies that
\begin{eqnarray*}
B_2^2(s)&=&\max_{r \in [0,L]} b_2(r)=(\delta-\bar r)^2\left(1+(\delta-\bar r)\dfrac{k(s)}{2}\right)^2\dfrac{1}{(1+(\delta-\bar r)k(s))^2},
\end{eqnarray*}
and hence
$$
C_2(s) \ge \dfrac{(1+\delta k(s))^2}{4\delta^2}.
$$
Combining this inequality with \eqref{bar2} implies
$$
\dint_0^\delta \psi_r^2 (1+rk(s))dr \ge \dfrac{(1+\delta k(s))^2}{ 4\delta^2} \dint_0^\delta \psi^2 (1+ r k(s))dr;
$$
hence from \eqref{Fubini1} we deduce that
\begin{equation}\label{lb3}
\mu_1(D) \ge \min_{s\in [0,L]}\dfrac{(1+\delta k(s))^2}{4\delta^2}.
\end{equation}

If we are in subcase $(ii)$ or $(iii)$, we note that if $r=0$, then $\psi=0$ for all $s \in [0,L]$.  We define $B_1^2(s)$ as in \eqref{b1squared} and follow that argument through \eqref{b1}.  Then \eqref{b11} may be replaced by
$$
B_1^2(s) \le \dfrac{\delta^2}{(1+\delta k(s))^2}
$$
and we conclude as in subcase $(i)$.

For case (3), we combine cases (1) and (2). We note that if $k(s)=0$ on a set $I \subseteq [0,L]$ with positive measure, then \eqref{iii} holds true for every $s \in I$. 
\end{proof}

\begin{remark}
In \cite{PW}, Payne and Weinberger establish a lower bound for $\mu_1(D)$ when $D$ is convex. As an application, they obtain a pointwise estimate for the solution $u$ to the interior Neumann problem:
\begin{equation}\nonumber
\left\{
\begin{array}{ll}
\Delta u=0 &\mbox{in}\>D,
\\ \\
\frac{\partial u}{\partial \mathbf{n}}=g &\mbox{on}\> \partial D,
\end{array}
\right.
\end{equation}
in terms of square integrals of $\frac{\partial u}{\partial \mathbf{n}}$ and $\mu_1(D)$. The argument of \cite{PW} that leads to the pointwise estimate holds for non-convex domains, so combining Theorem \ref{1} and Proposition \ref{2d-1} with the argument of \cite{PW} yields analogous pointwise estimates for domains $D$ satisfying the hypotheses of Theorem \ref{1} and such that $\mu_1(D) = \mu_1^{odd}(D)$.

\end{remark}

\begin{remark}
Let $\gamma$ and $D$ be as in Proposition \ref{2d-1}, assuming in particular that one of the conditions \eqref{delta}, \eqref{delta1} or \eqref{delta2} is fulfilled.   Then we claim that $\mu_1(D)$ is simple.  Indeed, we may view $\mu_1^{odd}(D)$ as the smallest eigenvalue of a mixed Dirichlet-Neumann eigenvalue problem with Dirichlet boundary conditions on $\{x=0\} \cap (D \cup \partial D)$ and Neumann boundary conditions on $\{x>0\} \cap \partial D$.  The smallest eigenvalue of such a mixed problem is always simple, so there exists a unique (up to a multiplicative constant) odd eigenfunction whose Rayleigh quotient is equal to  $\mu_1^{odd}(D)$.  If $\mu_1(D)=\mu_1^{odd}(D)$ were not simple, then we could create an even eigenfunction corresponding to $\mu_1(D)$ as done at the beginning of the proof of Prop. \ref{2d-1} and reach a contradiction.  
\end{remark}

We have thus established various conditions under which $\mu_1(D) = \mu_1^{odd}(D)$.  For a given domain $D$, we may combine a relevant condition with the lower bound on $\mu_1^{odd}(D)$ given by Theorem \ref{1} to give an explicit and easily computable lower bound on $\mu_1(D)$. We illustrate this idea with several examples.

\medskip

\begin{example}[Annular sector]
Let $R>0$ and consider the annular sector 
$$D=\{\rho e^{i\theta}: \> R < \rho < R+\delta, \> \frac{\pi}{2}-\alpha < \theta <\frac{\pi}{2}+\alpha\},$$ 
with $\alpha \in (0,\pi)$.  
Then condition \eqref{delta} in Proposition \ref{2d-1} becomes 
$$
(2R\delta + \delta^2) \left(\log \left(\frac{R+\delta}{R}\right)\right) < \frac{2\alpha^2R^2}{\pi^2}.
$$
Thus, if $\delta$ satisfies this inequality, Theorem \ref{1} tells us that
$$
\mu_1(D)=\mu_1^{odd}(D) \geq \frac{ \pi^2}{4\alpha^2(R+\delta)^2}.
$$
Note that we have given a lower bound on $\mu_1(D) = \mu_1^{odd}(D)$ without any reference to zeros of Bessel functions.
\end{example}

\begin{example}[Arch of catenary]
 Let $a >0$ and consider the arch of catenary $\gamma(s)=(x(s),y(s))$, $s\in [0,2 \sinh a]$, where
$$
\left\{\begin{array}{ll}
x(s)&=\arcsinh (s-\sinh a) \\
y(s)&=\sqrt{1+(s-\sinh a)^2}.
\end{array}
\right.
$$
We find $k(s)=\frac{1}{1+(s-\sinh a)^2}$, and $k$ is concave if $a\le \arcsinh\left(\frac{1}{\sqrt{3}}\right)$.  Since $k(s)>0$ for all $s$, any positive $\delta$ will satisfy the constraint given by Theorem \ref{1}.  If $\delta$ is small enough for condition \eqref{so} in Proposition \ref{2d} to hold,  Theorem \ref{1} gives the following explicit bound:
$$
\mu_1(D)=\mu_1^{odd}(D)\ge \frac{\pi^2}{4(1+\delta)^2\sinh^2 a}.
$$
\end{example}

\begin{example}[Handlebar moustache]
We begin by considering the following concave function on the interval $[0,1.6]$:
$$
k(s)=\left\{
\begin{array}{ll}
50 s-25 & \mbox{if}\> \> 0 \le s \le 0.6, \\
5 & \mbox{if}\> \>0.6 \le s \le 1, \\
-50 s+55 & \mbox{if} \>\> 1\le s \le 1.6.
\end{array}
\right.
$$
Up to a rotation and a translation, there exists a unique curve $\gamma (s)=(x(s),y(s))$ (parametrized with respect to its arc length) having curvature $k(s)$.  If $F(u)=\dint_0^u k(t)dt$, $u \in [0,1.6]$, then $\gamma$ has the following parametrization:
$$
\gamma(s)=\left(\int_0^s \cos F(u)\,du, \int_0^s \sin F(u)\,du\right),\quad 0\leq s \leq 1.6.
$$
 \begin{figure}[h]\label{fig}
\centering
 \includegraphics[scale=0.35]{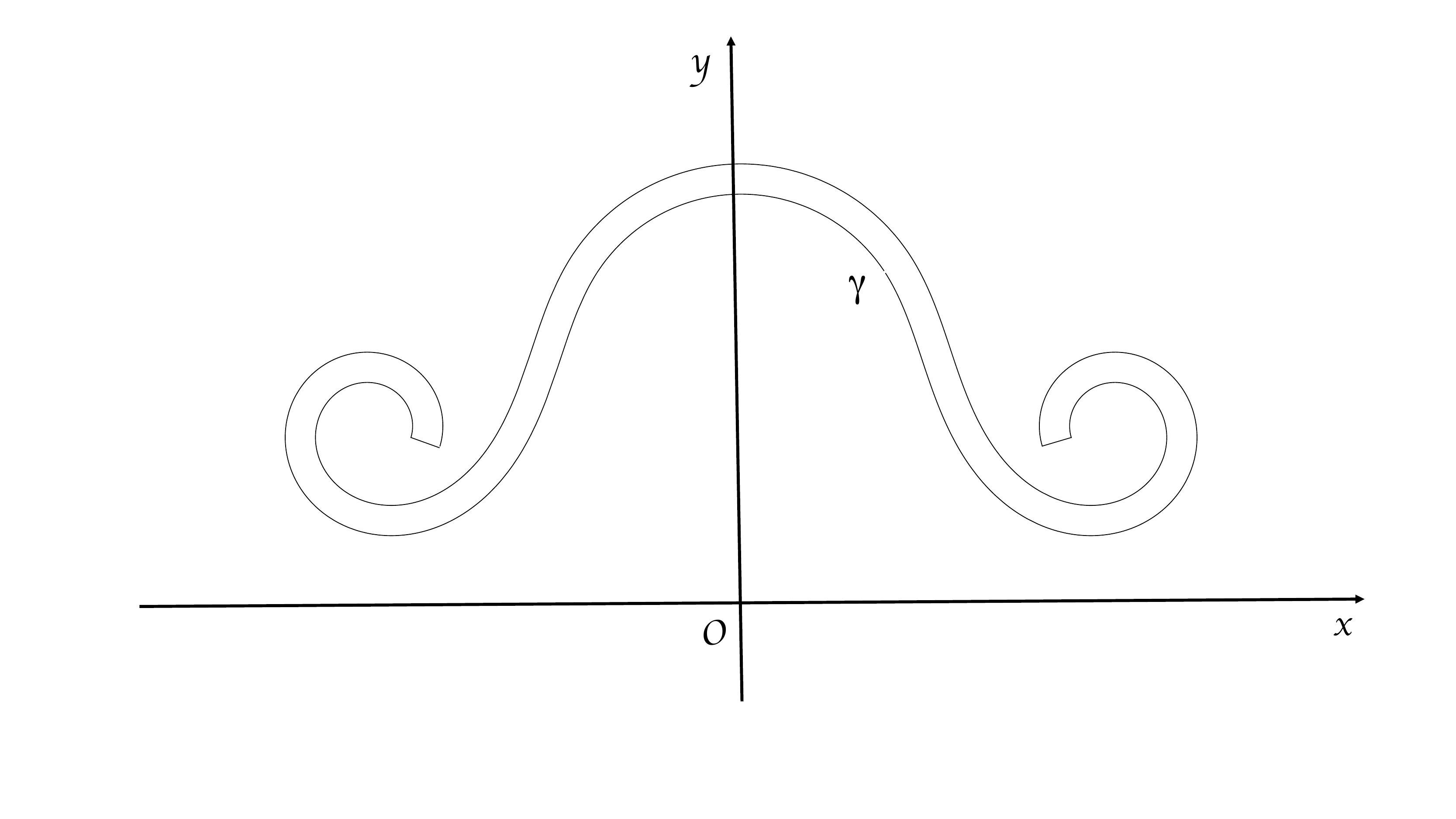}
\caption{The ``handlebar moustache'' domain of Example 3 with $\delta=.03$.}
\end{figure}
By rotating and translating so that $\gamma$ is symmetric with respect to the $y$-axis, we may build $D$ as in Theorem \ref{1} (see Figure $2$).

We next find positive values of $\delta$ so that \eqref{delta2} is satisfied. First observe that the requirement $1+\delta k(s)>0$ on $[0,1.6]$ forces $\delta<0.04$.  Next, note that
\[
\max_{s\in [0,1.6]} \delta^2(2+\delta k(s))^2=\delta^2(2+5\delta)^2\quad \textup{and}\quad \max_{s\in[0,1.6]}\frac{4\delta^2}{(1+\delta k(s))^2}=\frac{4\delta^2}{(1-25\delta)^2}.
\]
Since $1-25\delta>0$, the inequality
\[
\frac{4\delta^2}{(1-25\delta)^2}>\delta^2(2+5\delta)^2
\]
is equivalent to $2>(1-25\delta)(2+5\delta)=2-45\delta-125\delta^2$, which holds since $\delta>0$.
It follows that inequality \eqref{delta2} in Proposition \ref{2d-1} holds precisely when
\[
Q(\delta)=\frac{1.6^2}{\pi^2}\frac{\dint_0^{1.6}\cos^2\left(\frac{\pi}{1.6}s\right)\left(\delta+\frac{\delta^2}{2}k(s)\right)ds}{\dint_0^{1.6}\sin^2\left(\frac{\pi}{1.6}s\right)\frac{\log(1+\delta k(s))}{k(s)}ds}-\frac{4\delta^2}{(1-25\delta)^2}>0.
\]
We graph $Q(\delta)$ in Figure $3$ below
 \begin{figure}[h]\label{dg}
\centering
 \includegraphics[scale=.8]{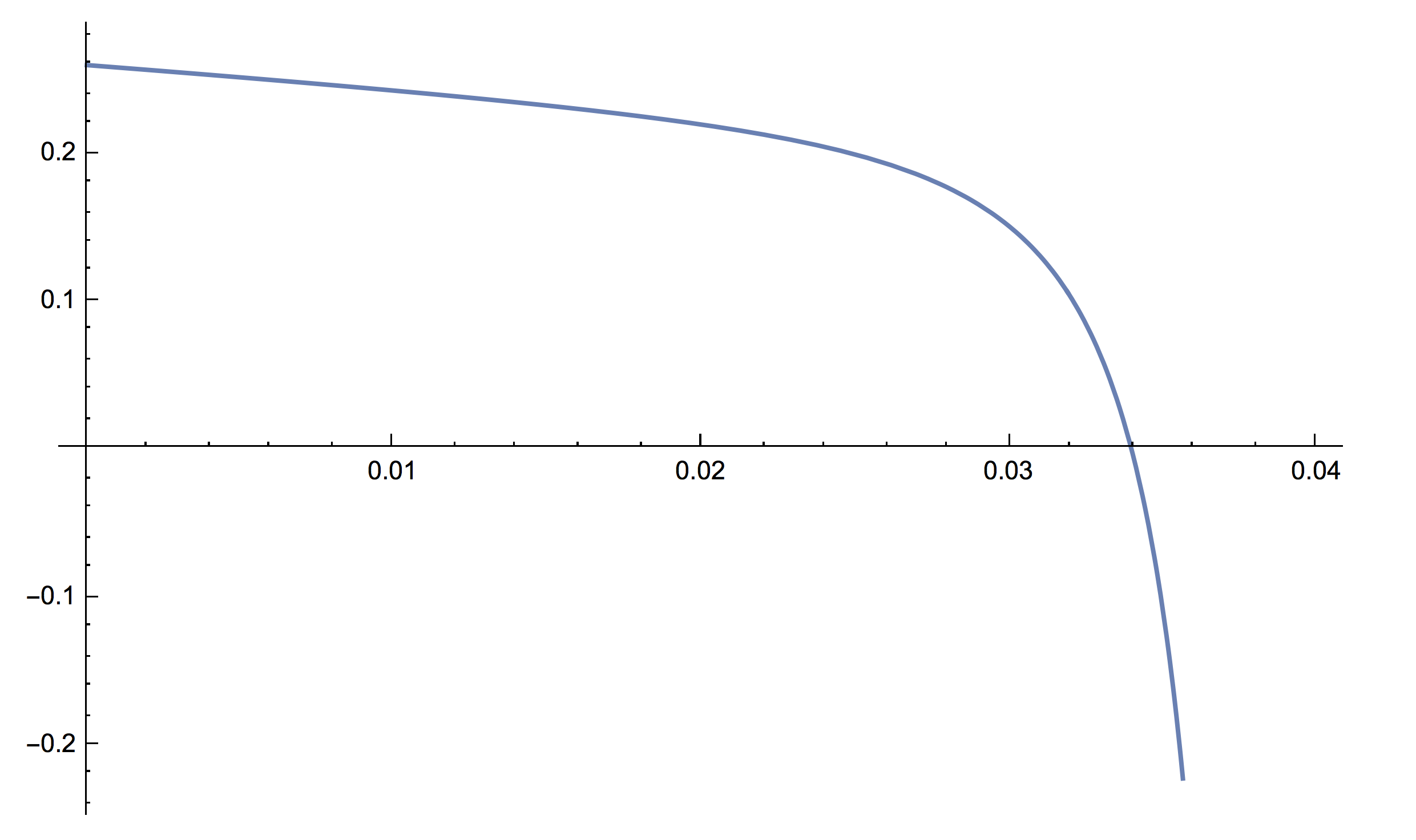}
\caption{A graph of $Q(\delta)$ from Example 3.}
\end{figure}
and find using Mathematica that $Q(\delta)>0$  provided $\delta<.03393$. We therefore have
$$
\mu_1(D)=\mu_1^{odd}(D) \ge \frac{\pi^2}{2.56(1+5\delta)^2}
$$
for such values of $\delta$.

\end{example}

%--------------------------------------------------------------------------------------------------------------------------------------------------------------------------------------------------------------------------------------------------------------------------------------------------------------------------------------------------------------------------------------------------------------

\section{Some considerations in the three-dimensional case}

Let $\Omega$ be a bounded subset of $\R^2$ and let $\varphi(s,t)=\left(x(s,t),y(s,t),z(s,t)\right)$, $(s,t)\in \Omega$, be a smooth surface. 
Consider the three-dimensional domain $D$ consisting of the points on one side of $\varphi$, within a suitable distance $\delta$ of $\varphi$. If we denote by $\nu=(\nu_x,\nu_y,\nu_z)=\dfrac{\varphi_s \times \varphi_t}{||\varphi_s \times \varphi_t||}$ a chosen unit normal vector to $\varphi$, $D$ can be described as follows:
\begin{equation} \label{3dD}
D=\left\{\left(x(s,t)+r\nu_x(s,t),y(s,t)+r\nu_y(s,t),z(s,t)+r\nu_z(s,t)\right): \> (s,t)\in \Omega, \> r \in (0,\delta) \right\}.
\end{equation}
In order to evaluate integrals over $D$, we introduce a Fermi coordinate system using $r=\dist_\varphi(x,y,z)$ as one coordinate and $s,t$, the coordinates in $\Omega$ of the point on $\varphi$ nearest to $(x,y,z)$, as the other ones. 
The domains $D$ that we will consider arise from smooth surfaces $\varphi$ that are generalized cylinders and surfaces of revolution.

\subsection{Generalized cylinders}\label{sub:gencyl}
A generalized cylinder is a special case of a ruled surface, which is a surface that is a union of straight lines.  We have a generalized cylinder when the straight lines, or \emph{rulings}, are all parallel to each other.  Given a set of parallel rulings, a parametrized curve $\alpha$ in $\mathbb{R}^3$ that meets each of these rulings, and a constant unit vector $\beta$ that is parallel to the rulings, the corresponding generalized cylinder may be described as 
\begin{equation} \label{gencyl}
\varphi(s,t) = \alpha(s) +t \beta, \qquad  t \in \R.
\end{equation} 
It can be shown that we may always assume that $\alpha$ is parametrized with respect to arc length and contained in a plane that is perpendicular to $\beta$.  Without loss of generality, suppose that $\alpha(s)=(x(s),0,z(s))$, $s\in [0,L]$, is a smooth, non-closed, simple curve, parametrized with respect to its arc length, whose image is contained in the plane $y=0$. Moreover, suppose $\alpha(s)$ is symmetric with respect to the $z$-axis so that
$$
x(L-s)=-x(s), \ \  z(L-s)=z(s), \quad s \in \left[0,\frac{L}{2}\right].
$$
Take $\beta = (0,-1,0)$ and consider the surface $\varphi$ given by \eqref{gencyl} with $s \in (0,L)$ and $t \in (0,T)$ for some $T>0$.  A typical domain $D$ constructed from a generalized cylinder as in \eqref{3dD} is shown in Figure $4$.

 \begin{figure}[h]\label{fig}
\centering
 \includegraphics[scale=0.4]{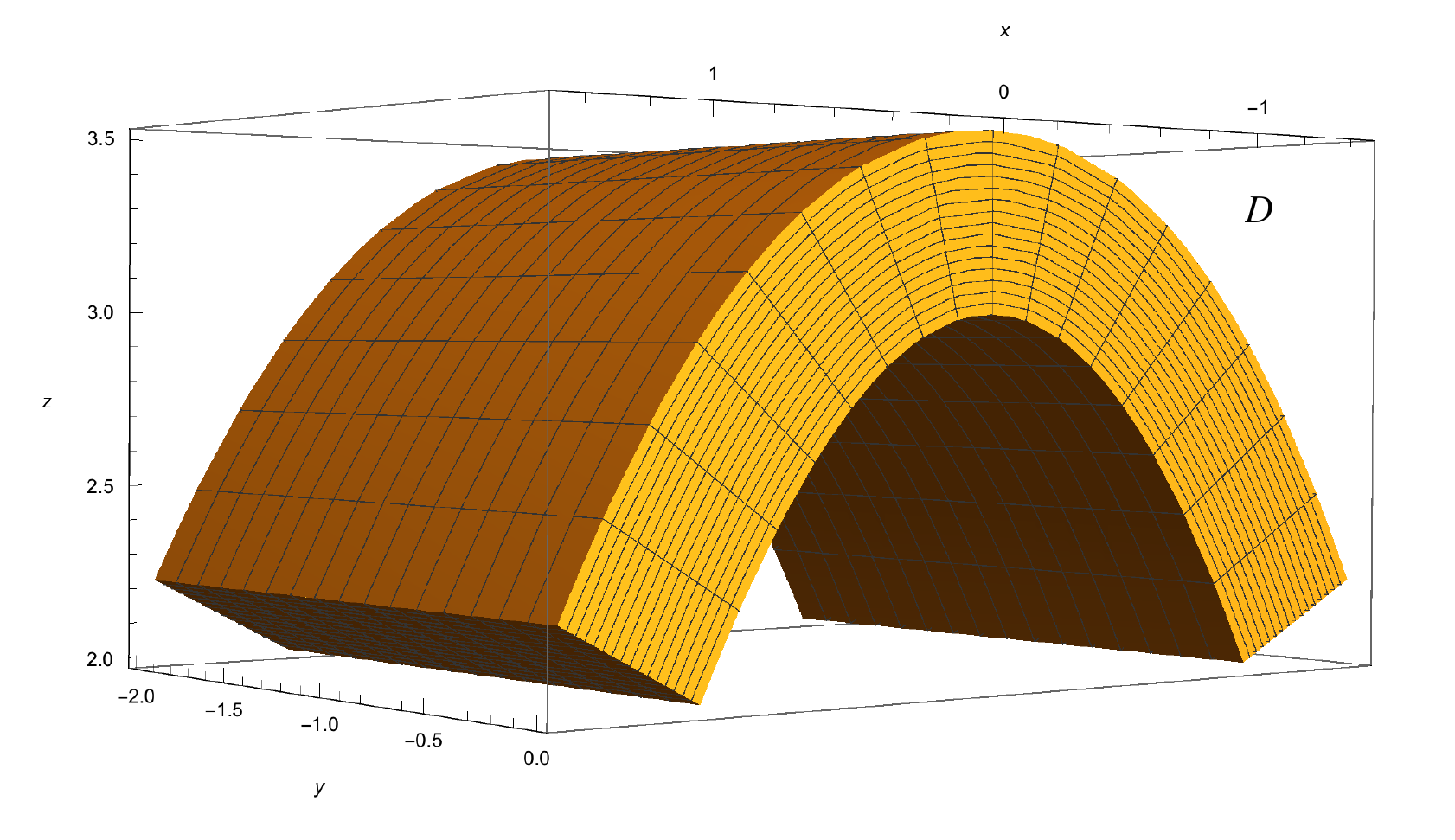}
\caption{A typical domain $D$ constructed from a generalized cylinder.}
\end{figure}

We see that $\{ \varphi_s, \varphi_t, \nu\}$ forms an orthonormal basis for $\mathbb{R}^3$.  Recalling notation from the appendix, we compute
\begin{eqnarray*}
E=|\varphi_s|^2 = 1, &\quad& F=\varphi_s\cdot \varphi_t=0, \quad  \quad G=|\varphi_t|^2 = 1, \\
a = \frac{FM-LG}{EG-F^2} = -L = - \varphi_{ss}\cdot \nu= k(s), &\quad& b=\frac{LF-EM}{EG-F^2} = -M = -\varphi_{st} \cdot \nu = 0, \\
c = \frac{FN-MG}{EG-F^2} = 0, &\quad& d=\frac{MF-NE}{EG-F^2} = -N = -\varphi_{t t}\cdot \nu=0, 
\end{eqnarray*}
where $k(s)$ is the curvature of $\alpha$.  Moreover, the Jacobian of the Fermi transformation is independent of $t$:
$$
\det(J(r,s,t)) =   \det(J(r,s))= 1+ar = 1+rk(s).
$$

With this setup, we can now give a lower bound on $\mu_1^{odd}$ for generalized cylinders.
\begin{theorem}\label{gc}
Suppose that the curvature $k(s)$ of $\alpha$ is concave in $[0,L]$ and let $\delta >0$ be such that  $1+\delta k(s) > 0$ in $[0,L]$. If $D$ is simply connected and $\mu_1^{odd}(D)$ is the smallest nontrivial Neumann eigenvalue having a corresponding eigenfunction that is odd with respect to the plane $x=0$, we have
$$
\mu_1^{odd}(D) \ge B \frac{\pi^2}{L^2},
$$
where $B=\underset{[0,\delta]\times[0,L]}{\min}\frac{1}{(1+rk(s))^2}$.
\end{theorem}

\begin{proof}
We will argue as in the two-dimensional case. Fix $n \in \mathbb{N}$. For any $i=0,\dots,n-1$,  and any $j=0,\dots, n-1$, let us denote by
$$
D_{ij}=\left\{(x,y,z)\in D:\> \frac{i\delta}{n}<\dist_\varphi(x,y,z)<\frac{(i+1)\delta}{n},\> \frac{jT}{n}<t<\frac{(j+1)T}{n}\right\}.
$$

\noindent Let $u$ be an eigenfunction corresponding to $\mu_1^{odd}(D)$ that is odd with respect to the plane $x=0$. Using the definition of eigenfunction and a Green's formula, we see that
$$
\mu_1^{odd}(D)=\frac{\dint_D |\nabla u|^2 dxdydz}{\dint_D u^2 dxdydz};
$$
moreover, the fact that $u$ is odd implies
$$ 
\int_{D_{ij}} u\,dxdydz=0 \quad  \text{for all } i=\ 0,\dots,n-1  \text{ and for all }j=0,\dots,n-1.
$$

We want to evaluate the energy of $u$ in any $D_{ij}$. Using the Fermi coordinate system
and denoting $\mathcal{J} =| \det(J) |$, we have
\begin{eqnarray*}\label{Di}
\int_{D_{ij}} |\nabla u|^2 dxdydz &=& \int_0^L\left(\int_{\frac{jT}{n}}^{\frac{(j+1)T}{n}}\int_{\frac{i\delta}{n}}^{\frac{(i+1)\delta}{n}}\left(u_r^2+u_s^2 \frac{(1+dr)^2}{\J^2}+u_t^2\frac{(1+ar)^2}{\J^2}\right)\, \J\, dr dt\right)ds
\\
&\ge & \min_{D_{ij}}\J^{-2}\int_0^L\left(\int_{\frac{jT}{n}}^{\frac{(j+1)T}{n}}\int_{\frac{i\delta}{n}}^{\frac{(i+1)\delta}{n}}u_s^2\,  \J\, dr dt\right)ds \notag
\\
&\ge&B \int_0^L\left(\int_{\frac{jT}{n}}^{\frac{(j+1)T}{n}}\int_{\frac{i\delta}{n}}^{\frac{(i+1)\delta}{n}}u_s^2\,  \J\, dr dt \right)ds. \notag
\end{eqnarray*}
Let us write
$$
\int_0^L\left(\int_{\frac{jT}{n}}^{\frac{(j+1)T}{n}}\int_{\frac{i\delta}{n}}^{\frac{(i+1)\delta}{n}}u_s^2 \,  \J\, dr dt\right)ds=I_1+I_2,
$$
where
\begin{eqnarray*}
I_1&=&\int_0^L\left(\int_{\frac{jT}{n}}^{\frac{(j+1)T}{n}}\int_{\frac{i\delta}{n}}^{\frac{(i+1)\delta}{n}}\left(u_s^2( r,s,t)-u_s^2\left( \frac{i\delta}{n},s,\frac{jT}{n}\right)\right) \J(r,s,t)dr dt\right)ds,\\
I_2&=&\int_0^L\left(\int_{\frac{jT}{n}}^{\frac{(j+1)T}{n}}\int_{\frac{i\delta}{n}}^{\frac{(i+1)\delta}{n}}u_s^2\left(\frac{i\delta}{n},s,\frac{jT}{n}\right) \J(r,s,t)dr dt\right)ds \\
&=&\frac{T\delta}{n^2}\int_0^L u_s^2\left( \frac{i\delta}{n},s,\frac{jT}{n}\right)\left(1+\frac{1+2i}{2} \frac{\delta}{n}k(s)\right)ds.
\end{eqnarray*}
Analogously, it holds that
$$
\int_{D_{ij}} u^2dxdydz=\int_0^L\left(\int_{\frac{jT}{n}}^{\frac{(j+1)T}{n}}\int_{\frac{i\delta}{n}}^{\frac{(i+1)\delta}{n}}u^2\, \J dr dt\right)ds=H_1+H_2,
$$
where
\begin{eqnarray*}
H_1&=&\int_0^L\left(\int_{\frac{jT}{n}}^{\frac{(j+1)T}{n}}\int_{\frac{i\delta}{n}}^{\frac{(i+1)\delta}{n}}\left(u^2( r,s,t)-u^2\left( \frac{i\delta}{n},s,\frac{jT}{n}\right)\right) \J(r,s,t) dr dt\right)ds,
\\
H_2&=&\int_0^L\left(\int_{\frac{jT}{n}}^{\frac{(j+1)T}{n}}\int_{\frac{i\delta}{n}}^{\frac{(i+1)\delta}{n}}u^2\left( \frac{i\delta}{n},s,\frac{jT}{n}\right)\J(r,s,t) dr dt\right)ds
\\
&=&\frac{T\delta}{n^2}\int_0^L u^2\left( \frac{i\delta}{n},s,\frac{jT}{n}\right)\left(1+\frac{1+2i}{2} \frac{\delta}{n} k(s)\right)ds.
\end{eqnarray*}
Let $A$ be a common bound for the absolute value of each of $u$ and its first and second derivatives when expressed in Fermi coordinates.  Applying the Mean Value Theorem, we deduce that
$$
|I_1| \le \frac{2A^2(T+\delta)}{n}|D_{ij}|, \quad 
|H_1|\le \frac{2A^2(T+\delta)}{n}|D_{ij}|.
$$
Using that $k(s)$ is even and $u$ is odd with respect to $\frac{L}{2}$, we see that
$$
\int_0^Lu\left(\frac{i\,\delta}{n}, s,\frac{jT}{n}\right)\left(1+\frac{1+2i}{2} \frac \delta n k(s)\right)ds=0.
$$
 Thus, arguing as in the two-dimensional case, we may apply Lemma \ref{2} to conclude that $I_2 \geq \dfrac{\pi^2}{L^2}H_2$.  We combine our estimates in a manner parallel to that of the two-dimensional case, summing over $i$ and $j$; taking the limit as $n$ goes to $+\infty$ yields the result.
\end{proof}

\begin{remark}
Let $D$ be constructed from a generalized cylinder as in Theorem $\ref{gc}$. By separation of variables, the eigenvalues of $D$ take the form
\[
\mu(D)=\mu_m(\widetilde{D})+\frac{n^2\pi^2}{T^2},\quad m,n\geq0,
\]
where $\widetilde{D}$ is a two-dimensional domain in the $xz$-plane as in \eqref{D}. If $T$ is sufficiently large, then $\mu_1(D)=\frac{\pi^2}{T^2}$ with corresponding eigenfunction $u(x,y,z)=\cos\left(\frac{\pi y}{T}\right)$. Observe that $u$ is odd with respect to the plane $y=-\frac{T}{2}$. Thus as $T$ becomes large, we expect eigenfunctions for $\mu_1(D)$ to exhibit odd symmetry with respect to the plane $y=-\frac T 2 $ rather than the plane $x=0$. On the other hand, if $T$ is sufficiently small, then $\mu_1(D)=\mu_1(\widetilde{D})$ and eigenfunctions for $\mu_1(D)$ take the form $u(x,y,z)=v(x,z)$, where $v$ is an eigenfunction for $\mu_1(\widetilde{D})$. Hence we may apply Proposition \ref{2d-1} to give conditions on $\delta$ that guarantee $\mu_1(\widetilde{D})=\mu_1^{odd}(\widetilde{D})$ and therefore $\mu_1(D)=\mu_1^{odd}(D)$.
\begin{comment}
Of course, if $T$ is too large compared to $L$, we cannot expect that $\mu_1(D)$ coincides with $\mu_1^{odd}(D)$.
\end{comment}
\end{remark}

\subsection{Surfaces of revolution}

Let $\alpha(s) = (x(s), 0, z(s)), s \in [0,L]$, be the curve considered in \S \ref{sub:gencyl}. We assume, without loss of generality, that $x(s)\geq 0$ for $s\in \left[0,\frac{L}{2}\right]$. Our aim is to construct a three-dimensional domain $D$ consisting of certain points on one side of the surface of revolution obtained by a  $\pi$ rotation of $\alpha$ around the $z$-axis. 
We consider the surface
\begin{equation}\label{phi+}
\varphi(s,\theta)=\left(x(s,\theta), y(s,\theta), z(s,\theta)\right)=\left(x(s)\sin \theta, x(s)\cos\theta,z(s)\right),\qquad s\in {\left(0,\frac L 2 \right), \> \theta \in (0, \pi)},
\end{equation}
and the domain
$$
D_+=\left\{\left(x(s,\theta)+r\nu_x(s,\theta),y(s,\theta)+r\nu_y(s,\theta),z(s,\theta)+r\nu_z(s,\theta)\right): \> (s,\theta)\in \left(0,\frac L 2\right)\times (0, \pi), \> r \in (0,\delta)\right\}.
$$
Then, $D=D_+\cup D_-\cup (\textup{int}(\partial D_+ \cap \{x=0\}))$, where $D_-$ is the domain symmetric to $D_+$ with respect to the $yz$-plane, and $\textup{int}$ denotes the two-dimensional interior taken in the plane $\{x=0\}$.
 Note that the parametrization of $\varphi$ is not regular at $s=\frac{L}{2}$ or at $s=0$; we identify all points corresponding to each such $s$-value and define $\nu$ at those points by continuous extension. Had we done a rotation through $\pi$ of the whole curve $\alpha$, we would have more serious issues with regularity. A typical domain $D$ constructed from a surface of revolution is shown in Figure $5$.

 \begin{figure}[h]\label{fig}
\centering
 \includegraphics[scale=0.6]{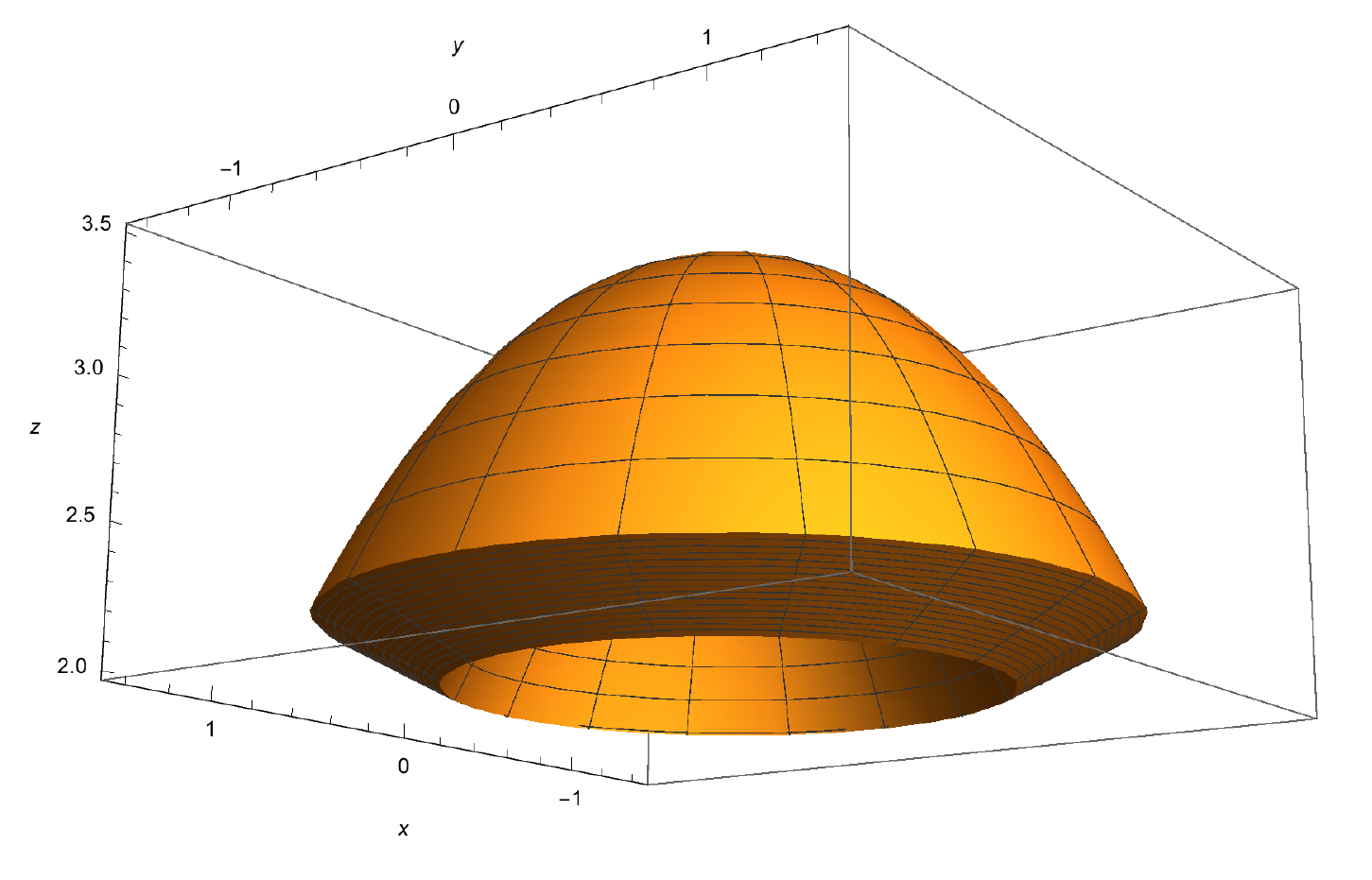}
\caption{A typical domain $D$ constructed from a surface of revolution.}
\end{figure}

We see that $\{ \varphi_s, \varphi_\theta, \nu\}$ forms an orthonormal basis for $\mathbb{R}^3$.  Recalling notation from the appendix, we compute
\begin{eqnarray*}
E=|\varphi_s|^2 = 1, \quad F&=&\varphi_s\cdot \varphi_\theta=0, \quad G=|\varphi_\theta|^2 = x(s)^2, \\
a = \frac{FM-LG}{EG-F^2} = -L = - \varphi_{ss}\cdot \nu= k(s), &\quad& b=\frac{LF-EM}{EG-F^2} = -\dfrac{M}{G} = -\dfrac{\varphi_{s\theta} \cdot \nu}{G} = 0, \\
c = \frac{FN-MG}{EG-F^2} = -M=-\varphi_{s\theta}\cdot\nu=0, &\quad& d=\frac{MF-NE}{EG-F^2} = -\dfrac{N}{G} = -\dfrac{\varphi_{\theta \theta}\cdot \nu}{G}=\dfrac{z'(s)}{x(s)}, 
\end{eqnarray*}
where $k(s)$ is the curvature of $\alpha$.  Moreover, the Jacobian of the Fermi transformation is independent of $\theta$:
$$
\det (J(r,s,\theta))=\det(J(r,s))=(x(s)+rz'(s))(1+rk(s)).
$$

With this setup, we may now give a lower bound on $\mu_1^{odd}$ for surfaces of revolution.

\begin{theorem}\label{33}
Let $\delta>0$ be such that $\det(J(r,s))>0$ on $[0,\delta] \times \left(0,\dfrac{L}{2}\right)$ and $1+\delta k(s)>0$ for $s\in \left[0,L \right]$. 
Suppose that for each $r\in[0,\delta]$, the function $\mathcal{J}(r,s)=|\det(J(r,s))|$ is concave in $s\in \left[0,\dfrac L 2 \right]$. Assume $D$ is simply connected and denote by $\mu_1^{odd}(D)$ the smallest nontrivial Neumann eigenvalue with a corresponding eigenfunction that is odd with respect to the plane $x=0$. Then
\begin{equation}\label{M3}
\mu_1^{odd}(D)\ge B \frac{\pi^2}{L^2},
\end{equation}
with $B= \displaystyle\min_{[0,\delta]\times [0,L]}\dfrac{1}{(1+rk(s))^2}.$
\end{theorem}

In order to prove this theorem we need a variant of Lemma \ref{2}.
\begin{lemma}\label{pw1}
Let $p(s)$ be a concave,  non-negative function on the interval $\left[0,\frac L 2\right]$ such that $p\left(\dfrac L 2 \right)=0$. Then for any piecewise twice differentiable function $v(s)$ that satisfies
$$
v'(0)=v\left(\dfrac L 2 \right)=0,
$$
it follows that
$$
\int_0^{{\frac L 2}} (v'(s))^2 p(s)ds \ge \frac{4\pi^2}{L^2}\int_0^{\frac L 2} (v(s))^2 p(s)ds.
$$
\end{lemma}

The proof of Lemma \ref{pw1} is similar to that of Lemma \ref{2} (cf. \cite{Beb}, \cite{PW}). Note that $v(s)$ satisfies a singular Sturm-Liouville problem and that we may make a change of variables as in the original proof:
$$
w=v'p^{1/2}.
$$ 
Since we are assuming that $v'(0)=v\left(\dfrac L 2 \right)=0$, we see that $w$ satisfies homogeneous Dirichlet boundary conditions and can thus serve as a test function in the Rayleigh quotient for a vibrating string of length $\dfrac{L}{2}$ with fixed ends.

\begin{proof}[Proof of Theorem \ref{33}]
In addition to Lemma \ref{pw1}, we will use a slight modification of the arguments in the proof of Theorem \ref{gc}. First, we observe that an eigenfunction $u$ corresponding to $\mu_1^{odd}(D)$ is the first eigenfunction of the Laplace operator on $D_+$ with mixed boundary conditions: Dirichlet on $\partial D_+\cap\{x=0\}$ and Neumann on the remaining part of $\partial D_+$. Fixing $n \in \N$, we partition $D_+$ as
$$
D_+^{ij}=\left\{(x,y,z)\in D_+:\> \frac{i\delta}{n}<\dist_\varphi(x,y,z)<\frac{(i+1)\delta}{n},\> \frac{j\pi}{n}<\theta<\frac{(j+1)\pi}{n}\right\}, \> 
$$
for $i=0,\dots,n-1$ and $j=0, \dots, n-1$. We observe that
\begin{eqnarray}\label{Di}
\int_{D_{ij}^+} |\nabla u|^2 dxdydz &=& \int_0^{\frac L 2}\left(\int_{\frac{j\pi}{n}}^{\frac{(j+1)\pi}{n}}\int_{\frac{i\delta}{n}}^{\frac{(i+1)\delta}{n}}\left(u_r^2+u_s^2 \frac{x^2(1+dr)^2}{\mathcal{J}^2}+u_\theta^2\frac{x^2(1+ar)^2}{\mathcal{J}^2}\right)\mathcal{J} dr d\theta \right)ds
\\
&\ge & \min_{D_{ij}^+}\frac{x^2(1+dr)^2}{\mathcal{J}^2}\int_0^{\frac L 2}\left(\int_{\frac{j\pi}{n}}^{\frac{(j+1)\pi}{n}}\int_{\frac{i\delta}{n}}^{\frac{(i+1)\delta}{n}}u_s^2\mathcal{J}\,dr d\theta\right)ds \notag
\\
&\ge&B\int_0^{\frac L 2}\left(\int_{\frac{j\pi}{n}}^{\frac{(j+1)\pi}{n}}\int_{\frac{i\delta}{n}}^{\frac{(i+1)\delta}{n}}u_s^2\mathcal{J}\,dr d\theta \right)ds, \notag
\end{eqnarray}
where in the last line, we have used that $\frac{x^2(1+dr)^2}{\mathcal{J}^2}=\frac{1}{(1+rk(s))^2}$. Let us write
$$
\int_0^{\frac L2}\left(\int_{\frac{j\pi}{n}}^{\frac{(j+1)\pi}{n}}\int_{\frac{i\delta}{n}}^{\frac{(i+1)\delta}{n}}u_s^2 \mathcal{J}\,dr d\theta \right)ds=I_1+I_2,
$$
where
\begin{eqnarray*}
I_1&=&\int_0^{\frac L 2}\left(\int_{\frac{j\pi}{n}}^{\frac{(j+1)\pi}{n}}\int_{\frac{i\delta}{n}}^{\frac{(i+1)\delta}{n}}\left(u_s^2(r, s, \theta)-u_s^2\left(\frac{i\delta}{n}, s,\frac{j\pi}{n}\right)\right)\mathcal{J}dr d\theta \right)ds,
\\
I_2&=&\int_0^{\frac L 2}\left(\int_{\frac{j\pi}{n}}^{\frac{(j+1)\pi}{n}}\int_{\frac{i\delta}{n}}^{\frac{(i+1)\delta}{n}}u_s^2\left(\frac{i\delta}{n}, s,\frac{j\pi}{n}\right)\mathcal{J}dr d\theta \right)ds\\
&=&\int_0^{\frac L 2}u_s^2\left(\frac{i\delta}{n}, s,\frac{j\pi}{n}\right)p(s)ds,
\end{eqnarray*}
with
\begin{equation} \label{Eqn:pDef}
p(s)=\int_{\frac{j\pi}{n}}^{\frac{(j+1)\pi}{n}}\int_{\frac{i\delta}{n}}^{\frac{(i+1)\delta}{n}}\mathcal{J}(r,s)\,dr d\theta .
\end{equation}

Analogously, it holds that
$$
\int_{D_{ij}^+} u^2dxdydz=\int_0^{\frac L 2}\left(\int_{\frac{j\pi}{n}}^{\frac{(j+1)\pi}{n}}\int_{\frac{i\delta}{n}}^{\frac{(i+1)\delta}{n}}u^2(r, s, \theta)\mathcal{J}\,dr d\theta \right)ds=H_1+H_2,
$$
where
\begin{eqnarray*}
H_1&=&\int_0^{\frac L 2}\left(\int_{\frac{j\pi}{n}}^{\frac{(j+1)\pi}{n}}\int_{\frac{i\delta}{n}}^{\frac{(i+1)\delta}{n}}\left(u^2(r, s, \theta)-u^2\left(\frac{i\delta}{n}, s,\frac{j\pi}{n}\right)\right)\mathcal{J} dr d\theta \right)ds,
\\
H_2&=&\int_0^{\frac L 2}\left(\int_{\frac{j\pi}{n}}^{\frac{(j+1)\pi}{n}}\int_{\frac{i\delta}{n}}^{\frac{(i+1)\delta}{n}}u^2\left(\frac{i\delta}{n}, s,\frac{j\pi}{n}\right)\mathcal{J}dr d\theta \right)ds
\\
&=&\int_0^{\frac L 2}u^2\left(\frac{i\delta}{n}, s,\frac{j\pi}{n}\right)p(s)ds.
\end{eqnarray*}
Since $\mathcal{J}(r,s)$ is concave in $s \in \left[ 0, \dfrac{L}{2} \right]$ and $\mathcal{J}\left(r,\dfrac{L}{2}\right)=0$ for each $r$, we see that $p(s)$ as defined in \eqref{Eqn:pDef} satisfies the hypotheses of Lemma \ref{pw1}. The lemma gives a relationship between $I_2$ and $H_2$, and the remainder of the proof follows that of Theorem \ref{gc}.
\end{proof}

\begin{example}[Half-spherical shell]
If $\alpha(s)=\left(R\sin\left(\dfrac s R\right), 0, -R \cos \left(\dfrac s R\right)\right)$, $s \in [0,2\pi R]$, then for any $\delta>0$ we get the half spherical shell 
$$
D_+=\left\{\left((r+R)\sin\left(\dfrac s R\right)\sin \theta,\>(r+R)\sin\left(\dfrac s R\right)\cos \theta,\>-(r+R)\cos\left(\dfrac s R\right)\right): \> r\in(0,\delta), s\in [0, \pi R], \theta \in (0,\pi)\right\}.
$$
In this case, we have
\[
\frac{1}{(1+rk(s))^2}=\frac{R^2}{(r+R)^2}.
\]
In \cite{L}, Li proves that $\mu_1(D)$ has multiplicity 3. Thus $\mu_1(D)$ must correspond to the angular eigenfunctions in the usual separation of variables, and hence there is an odd eigenfunction associated to $\mu_1(D)$.  Theorem \ref{33} gives
$$
\mu_1(D)=\mu_1^{odd}(D)\ge \frac{1}{4(\delta +R)^2}.$$
\end{example}

%--------------------------------------------------------------------------------------------------------------------------------------------------------------------------------------------------------------------------------------------------------------------------------------------------------------------------------------------------------------------------------------------------------------

\section{Appendix}

Here we provide some details about the Fermi coordinate systems in two and three dimensions. For the two-dimensional computations, let $\gamma(s) = (x(s), y(s)), \>s \in [0,L],$ be a smooth,  non-closed, simple curve, parametrized with respect to arc length.  Let $D$ be a simply connected domain described with coordinates $(r,s)$ as in \eqref{D}.  Consider the coordinates $X=X(r,s)$ and $Y=Y(r,s),$ where
\begin{align*}
X&=x(s)+ry'(s), \\
Y&=y(s)-rx'(s).
\end{align*}
Recall that the signed curvature $k(s)$ is defined by
\[
k(s)=x'(s)y''(s)-y'(s)x''(s).
\]
It is straightforward to verify that the Jacobian matrix is
\[
J = \frac{\partial(X,Y)}{\partial(r,s)}=\left(
\begin{matrix}
y' & x'+ry'' \\
-x' & y'-rx''
\end{matrix}\right).
\]
Assuming that $1+\delta k(s)>0$ for $s\in [0,L]$,  we see that the absolute value may be dropped:
\begin{equation}\label{inv}
\J =|\det (J)|=1+rk(s)>0 \quad \mbox{for every}\> s\in [0,L], \, r \in [0,\delta].
\end{equation}

We explicitly observe that \eqref{inv} and the Inverse Function Theorem imply that the Fermi coordinates from $[0,\delta]\times[0,L]$ to $\overline{D}$  are locally one-to-one. Moreover, since $D$ is simply connected, the Global Invertibility Theorem ensures that the Fermi coordinates are globally one-to-one. 
For smooth functions $u$, we have
\begin{equation}\label{co}
\int_{D}u(X,Y)dXdY=\int_0^L\int_0^{\delta}u(r,s)(1+rk(s))drds.
\end{equation}
Similarly, one calculates
\begin{align*}
u_X&=-\frac{1}{1+rk(s)}\left( -x'u_s+(rx''-y')u_r\right),\\
u_Y&=-\frac{1}{1+rk(s)}\left(-y'u_s+(x'+ry'')u_r\right),
\end{align*}
from which one deduces
\[
|\nabla u(X,Y)|^2=\frac{1}{(1+rk(s))^2}u_s^2+u_r^2.
\]

\medskip
In the three-dimensional case, let $\varphi$ be a surface described as $\varphi(s,t) = (x(s,t), y(s,t), z(s,t)), (s,t) \in \Omega$ with $\Omega$ a bounded domain in $\mathbb{R}^2$.  Let $D$ be a simply connected three-dimensional domain described with coordinates $(r,s,t)$ as in \eqref{3dD}, and suppose that $\{\varphi_s, \varphi_t, \nu \}$ forms an orthonormal basis for $\mathbb{R}^3$.  Consider the coordinates $X=X(r,s,t)$, $Y=Y(r,s,t)$ and $Z=Z(r,s,t)$, where
\begin{align*}
X&=x(s,t)+r\nu_x(s,t),\\
Y&=y(s,t)+r\nu_y(s,t), \\
Z&=z(s,t)+r\nu_z(s,t).
\end{align*}
The Jacobian matrix of this transformation is
\[
J = \frac{\partial(X,Y,Z)}{\partial(r,s,t)}=\left(
\begin{matrix}
\nu_x & x_s+r\nu_{x,s} & x_t+r\nu_{x,t} \\
\nu_y & y_s+r\nu_{y,s} & y_t+r\nu_{y,t} \\
\nu_z & z_s+r\nu_{z,s} & z_t+r\nu_{z,t}
\end{matrix}\right) .
\]

To simplify the computation of the determinant of this Jacobian, we recall some notation.  Let
\begin{eqnarray*}
E =|\varphi_s|^2, \quad  F&=\varphi_s\cdot \varphi_t,   \quad G&=|\varphi_t|^2,
\\ 
  L =\varphi_{ss}\cdot \nu,  \quad M&=\varphi_{s t}\cdot \nu,  \quad N&=\varphi_{t t}\cdot \nu,
\end{eqnarray*}

\begin{eqnarray}\label{notation}
a=\frac{FM-LG}{EG-F^2},& b=\dfrac{LF-EM}{EG-F^2}, \nonumber
\\ \\
c=\frac{FN-MG}{EG-F^2},& d=\dfrac{MF-NE}{EG-F^2}, \nonumber
\end{eqnarray}

\begin{eqnarray*}
&H&=-\dfrac{a+d}{2}=\dfrac{LG+NE-2MF}{2(EG-F^2)} \qquad \mbox{mean curvature,}
\\
&K&=ad-bc=\dfrac{LN-M^2}{EG-F^2} \qquad \mbox{Gauss curvature}.
\end{eqnarray*}
We can simplify 
\begin{eqnarray*}
\det(J) &=& \nu \cdot \left[ \left(\varphi_s + r \nu_s\right) \times \left(\varphi_t + r \nu_t \right)\right] \\
&=& \frac{\varphi_s \times \varphi_t}{||\varphi_s \times \varphi_t||}\cdot \left[ \left(\varphi_s + r \nu_s\right) \times \left(\varphi_t + r \nu_t\right)\right]
\end{eqnarray*}
using properties of the dot and cross products.  In addition to well-known properties, we use Lagrange's identity, which states that $(\mathbf{a} \times \mathbf{b}) \cdot (\mathbf{c} \times \mathbf{d}) = (\mathbf{a} \cdot \mathbf{c})(\mathbf{b} \cdot \mathbf{d}) - (\mathbf{a} \cdot \mathbf{d})(\mathbf{b} \cdot \mathbf{c})$ for vectors $\mathbf{a}, \mathbf{b}, \mathbf{c}, \mathbf{d}$.  We obtain
\[
 \J=|\det (J)| = \left(1-2rH+r^2K\right)\left(EG-F^2\right)^{1/2};
\]
here we have assumed $1-2rH+r^2K>0$ in order to drop the absolute values.
Thus for a smooth function $u$, we have
\begin{equation}\label{co}
\int_{D}u(X,Y,Z)dXdYdZ=\int_0^L\int_0^T\int_0^{\delta}u(r,s,t)\left(1-2rH+r^2K\right)\left(EG-F^2\right)^{1/2}dr dt ds.
\end{equation}

Finally, we need to express $|\nabla u(X,Y,Z)|^2$ with respect to Fermi coordinates.  We first compute
\begin{eqnarray*}
u_r &=& \nabla u \cdot \nu ,\\
u_s &=& \nabla u \cdot \varphi_s + r(\nabla u \cdot \nu_s), \\
u_t &=& \nabla u \cdot \varphi_t + r(\nabla u \cdot \nu_t).
\end{eqnarray*}
Then the reader may verify that, using the notation given in \eqref{notation},
\begin{eqnarray*}
\nu_s = a \varphi_s + b\varphi_t \quad \quad &\text{and}& \quad \quad \nu_t = c \varphi_s + d\varphi_t, \\ \\
\nabla u \cdot \varphi_s = \frac{(1+dr)u_s - bru_t}{1-2rH+r^2K} \quad &\text{and}& \quad \nabla u \cdot \varphi_t = \frac{(1+ar)u_t - cru_s}{1-2rH+r^2K}.
\end{eqnarray*}
Combining these expressions and making the additional assumption that $b=c=0$, we obtain
\[
|\nabla u|^2 =u_r^2+u_s^2\frac{(1+dr)^2}{(1-2rH+r^2K)^2}+u_t^2\frac{(1+ar)^2}{(1-2rH+r^2K)^2}.
\]

\noindent {\bf Acknowledgements}. B.B. and F.C. would like to thank the Department of Mathematics at Bucknell University for warm hospitality and support.  E.B.D. was partially supported by a grant from the Simons Foundation (210445). J.J.L. appreciates the funding he received from Bucknell University's International Research Travel Grants program. The authors would like to thank the anonymous referee for carefully reading the manuscript and for suggesting changes that strengthened the paper.

\medskip
%--------------------- 


\begin{thebibliography}{99}


\bibitem{AC}      M. S. Ashbaugh and F. Chiacchio, \textsl{On low eigenvalues of the Laplacian with mixed boundary conditions}, J. Differential Equations 250 (2011),
2544--2566. 

\bibitem{B}
C. Bandle, Isoperimetric inequalities and applications. 
Monographs and Studies in Mathematics, 7. Pitman (Advanced Publishing Program), Boston, Mass.--London, 1980. 

\bibitem{BB}
R. Ba$\tilde{\rm n}$uelos and K. Burdzy, 
\textsl{On the ``hot spots'' conjecture of J. Rauch}, J. Funct. Anal. 164 (1999), no. 1, 1--33.

%\bibitem{Bau}
%H. F. Bauer,
%\textsl{Tables of zeros of cross product Bessel functions}, Math. Comp. 18 1964 128--135. 

\bibitem{Beb}
M. Bebendorf,
\textsl{A note on the Poincar\'{e} inequality for convex domains}, Z. Anal. Anwendungen 22 (2003), no. 4, 751--756.

\bibitem{BCT_CPDE}
B. Brandolini, F. Chiacchio and C. Trombetti, 
\textsl{Sharp estimates for eigenfunctions of a Neumann problem}, Comm. Partial Differential Equations 34 (2009), no. 10-12, 1317--1337.

\bibitem{BCT_DIE}
B. Brandolini, F. Chiacchio and C. Trombetti, 
\textsl{A sharp lower bound for some Neumann eigenvalues of the Hermite operator}, Differential Integral Equations 26 (2013), no. 5-6, 639--654.

\bibitem{BCT}
B. Brandolini, F. Chiacchio and C. Trombetti, 
\textsl{Optimal lower bounds for eigenvalues of linear and nonlinear Neumann problems}, Proc. Roy. Soc. Edinburgh Sect. A 145 (2015), no. 1, 31--45.

\bibitem{BT}
A. Burchard and L. E. Thomas, 
\textsl{On an isoperimetric inequality for a Schr\"odinger operator depending on the curvature of a loop}, 
J. Geom. Anal. 15 (2005), no. 4, 543--563. 

\bibitem{CF}
L. A. Caffarelli and A. Friedman, 
\textsl{Partial regularity of the zero-set of solutions of linear and superlinear elliptic equations}, 
J. Differential Equations 60 (1985), no. 3, 420--433. 



\bibitem{EFK}
P. Exner, P. Freitas and D. Krej\v cir\'\i k, 
\textsl{A lower bound to the spectral threshold in curved tubes}, 
Proc. R. Soc. Lond. Ser. A Math. Phys. Eng. Sci. 460 (2004), no. 2052, 3457--3467. 

\bibitem{EHL}
 P. Exner, E. M. Harrell and M. Loss, 
 \textsl{Optimal eigenvalues for some Laplacians and Schr\"odinger operators depending on curvature}, Mathematical results in quantum mechanics (Prague, 1998), 47--58, Oper. Theory Adv. Appl., 108, Birkh\"auser, Basel, 1999.

\bibitem{FNT}
V. Ferone, C. Nitsch and C. Trombetti, \textsl{A remark on optimal weighted Poincar\'e inequalities for convex domains}, Atti Accad. Naz. Lincei Rend. Lincei Mat. Appl. 23 (2012), no. 4, 467--475.

\bibitem{GU}
V. Gol'dshtein and A. Ukhlov,
\textsl{On the first eigenvalues of free vibrating membranes in conformal regular domains}, Arch. Rational Mech. Anal.
 (DOI) 10.1007/s00205-016-0988-9.

\bibitem{H} 
A. Henrot,  Extremum problems for eigenvalues of elliptic operators. Frontiers in Mathematics. Birkh\"auser Verlag, Basel, 2006.

\bibitem{J}
 D. Jerison, 
 \textsl{Locating the first nodal line in the Neumann problem}, Trans. Amer. Math. Soc. 352 (2000), no. 5, 2301--2317.

\bibitem{L}
L. Li, \textsl{On the second eigenvalue of the Laplacian in an annulus}, Illinois J. Math. 51 (2007), no. 3, 913--925.

\bibitem{M}
 V. G. Maz'ja,  Sobolev spaces. Translated from the Russian by T. O. Shaposhnikova. Springer Series in Soviet Mathematics. Springer-Verlag, Berlin, 1985.

\bibitem{LP}
P. L. Lions and F. Pacella, \textsl{Isoperimetric inequalities for convex cones}, Proc. Amer. Math. Soc. 109  (1990), (2) 477--485.

\bibitem{PW}
 L.E. Payne and H. F. Weinberger,
 \textsl{An optimal Poincar\'{e} inequality for convex domains}, Arch. Rational Mech. Anal., 5 (1960), 286--292.


\bibitem{V}
D. Valtorta,
\textsl{Sharp estimate on the first eigenvalue of the $p$-Laplacian}, Nonlin. Analysis 75
(2012), 4974--4994.

\end{thebibliography}
\end{document}